\documentclass[12pt]{article}
\usepackage[overload]{textcase}

\usepackage{amsbsy,amsmath,amsthm,amssymb,graphicx}
\usepackage{natbib}
\RequirePackage[colorlinks,citecolor=blue,urlcolor=blue]{hyperref}
\usepackage{pstricks}
\usepackage{subcaption}
\usepackage{jhtitle}

\newcommand{\pcite}[1]{\citeauthor{#1}'s \citeyearpar{#1}}

\newcommand{\df}{\mathrm{d}}

\newtheorem{theorem}{Theorem}

\newtheorem{corollary}[theorem]{Corollary}
\newtheorem{remark}[theorem]{Remark}
\newtheorem{proposition}[theorem]{Proposition}

\newcommand{\X}{{\mathsf{X}}}

\newcommand{\citepage}[2]{\citeauthor{#1} (\citeyear{#1}, #2)}

\title{\bf On the limitations of single-step drift and minorization in
	Markov chain convergence analysis}

	\author{Qian Qin \\ School of Statistics \\ University of Minnesota 
	\and James P. Hobert \\ Department of Statistics \\ University of Florida}
	

\begin{document}
	
	\maketitle
		
		\begin{abstract}
			Over the last three decades, there has been a considerable effort within
			the applied probability community to develop techniques for bounding the
			convergence rates of general state space Markov chains.  Most of these
			results assume the existence of drift and minorization (d\&m)
			conditions.  It has often been observed that convergence rate bounds
			based on single-step d\&m tend to be overly conservative, especially in
			high-dimensional situations.  This article builds a framework for
			studying this phenomenon.  It is shown that any convergence rate bound
			based on a set of d\&m conditions cannot do better than a certain
			unknown optimal bound.  Strategies are designed to put bounds on the
			optimal bound itself, and this allows one to quantify the extent to
			which a d\&m-based convergence rate bound can be sharp.  The new theory
			is applied to several examples, including a Gaussian autoregressive
			process (whose true convergence rate is known), and a Metropolis
			adjusted Langevin algorithm.  The results strongly suggest that
			convergence rate bounds based on single-step d\&m conditions are quite
			inadequate in high-dimensional settings.
		\end{abstract}


	\section{Introduction}
	\label{sec:intro}
	
	The performance of a Markov chain Monte Carlo (MCMC) algorithm is
	directly tied to the convergence rate of the underlying Markov chain.
	(As will be made precise below, the convergence rate is a number
	between~$0$ and~$1$, with smaller values corresponding to faster
	convergence.)  Unfortunately, ascertaining the convergence rates of
	even mildly complex Markov chains can be extremely difficult.  Indeed,
	over the last three decades, there has been a considerable effort
	within the applied probability community to develop techniques for
	(upper) bounding the convergence rates of general state space Markov
	chains.  Most of these results assume the existence of drift and
	minorization (d\&m) conditions for the chain under study.  In essence,
	the minorization condition guarantees that the chain is well-behaved
	on a subset of its state space, and the drift condition guarantees
	that the chain will visit that subset frequently.  By carefully
	combining the d\&m, one can construct a quantitative upper bound on
	the chain's convergence rate that is an explicit function of the
	parameters in the d\&m conditions \citep[see,
	e.g.,][]{meyn1994computable,rosenthal1995minorization,roberts1999bounds,douc2004practical,baxendale2005renewal,jerison2019}.
	However, it is well known that d\&m-based bounds are often overly
	conservative; that is, the upper bound is often very close to~$1$,
	even when the chain is known to (or at least appears to) converge
	rapidly.  An example of this phenomenon is provided later in this
	section.  Worse yet, the problem is often exacerbated by increasing
	dimension \citep{rajaratnam2015mcmc,qin2018wasserstein}.  These facts
	raise the following question: Are d\&m-based methods inadequate for
	constructing sharp convergence rate bounds in high-dimensional
	problems?  This question is difficult to answer because, in situations
	where the methods fail, there are several vastly different potential
	reasons for the failure, including the possibility that the particular
	set of d\&m conditions that were used is somehow faulty.  In this
	article, we provide a partial answer to the question posed above by
	studying the \textit{optimal} bound that can be produced using a set
	of d\&m conditions.  When applied to specific Markov chains, our
	results strongly suggest that d\&m arguments based on single-step
	transition laws are quite inadequate in high-dimensional settings.  In
	the remainder of this section, we provide an overview of our results.
	
	Suppose that $(\X, \mathcal{B})$ is a countably generated measurable
	space, and let $P:\X \times \mathcal{B} \to [0,1]$ be a Markov
	transition kernel (Mtk).  When the state space,~$\X$, is a commonly
	studied topological space (e.g., a Euclidean space),~$\mathcal{B}$ is
	assumed to be its Borel $\sigma$-algebra.  For any positive
	integer~$m$, let $P^m$ be the $m$-step transition kernel, so that $P^1
	= P$.  For any probability measure $\mu: \mathcal{B} \to [0,1]$ and
	measurable function $f: \X \to \mathbb{R}$, denote $\int_{\X} \mu(\df
	x) P^m(x, \cdot)$ by $\mu P^m (\cdot)$, and $\int_{\X} P^m(\cdot, \df
	x) f(x)$ by $P^m f(\cdot)$.  Also, let $L^2(\mu)$ denote the set of
	measurable, real-valued functions on~$\X$ that are square integrable
	with respect to $\mu(\df x)$.
	
	For the time being, assume that the Markov chain defined by~$P$ has a
	stationary probability measure~$\Pi$, so $\Pi = \Pi P$.  The goal of
	convergence analysis is to understand how fast $\mu P^m$ converges
	to~$\Pi$ as~$m \to \infty$ for a large class of initial laws~$\mu$.  The
	difference between $\mu P^m$ and $\Pi$ is usually measured using the
	total variation distance, which is defined as follows.  For two
	probability measures on $(\X, \mathcal{B})$, $\mu$ and $\nu$, their
	total variation distance is
	\[
	d_{\mbox{\scriptsize{TV}}} (\mu,\nu) = \sup_{A \in \mathcal{B}} \,
	[\mu(A) - \nu(A)] \,.
	\]
	The Markov chain is \textit{geometrically ergodic} if, for each $x \in
	\X$, $d_{\mbox{\scriptsize{TV}}}(\delta_x P^m, \Pi)$ decays at an
	exponential rate that is independent of~$x$ when $m \to \infty$, where
	$\delta_x$ is the point mass (Dirac measure) at~$x$.  In other words,
	the chain is geometrically ergodic if there exist $\rho < 1$ and $M:
	\X \to [0,\infty)$ such that, for each $x \in \X$ and positive
	integer~$m$,
	\begin{equation}
	\label{eq:geometric}
	d_{\mbox{\scriptsize{TV}}}(\delta_x P^m, \Pi) \leq M(x) \rho^m \,.
	\end{equation}
	Following the ideas in \citepage{roberts2001geometric}{page 40}, define
	\[
	\rho_*(P) = \exp \bigg[ \sup_{x \in \X} \, \limsup\limits_{m \to \infty}
	\frac{\log d_{\mbox{\scriptsize{TV}}}(\delta_x P^m, \Pi)}{m} \bigg]
	\,.
	\]
	It can be shown that $\rho_*(P) \in [0,1]$.
	If~\eqref{eq:geometric} holds for all~$x$
	and~$m$, then $\rho_*(P) \leq \rho$.  On the other hand, if $\rho >
	\rho_*(P)$, then for each $x \in \X$, there exists $M_x > 0$ such
	that, for each $m > M_x$,
	\[
	\frac{\log d_{\mbox{\scriptsize{TV}}}(\delta_x P^m, \Pi)}{m} <
	\log \rho \,.
	\]
	In this case,~\eqref{eq:geometric} holds with $M(x) = \rho^{-M_x} + 1$.  
	Thus, $\rho_*(P)$ can be regarded as the
	true (geometric) convergence rate of the chain, and the chain is
	geometrically ergodic if and only if $\rho_*(P) < 1$.  Essentially, if
	$\rho_*(P)$ is close to~$0$, then the chain converges rapidly, and the
	elements of the chain are nearly independent; if $\rho_*(P)$ is close
	to~$1$, then the chain mixes slowly, and the elements of the chain are
	strongly correlated.
	
	The quantity $\rho_*(P)$ plays an important role in the
        analysis of MCMC algorithms \citep[see,
          e.g.,][]{jones2001honest,roberts2004general,rajaratnam2015mcmc}.
        However, despite its significance, it is typically quite
        difficult to get a handle on $\rho_*(P)$ (outside of toy
        problems).  As mentioned above, there are a number of
        different methods for converting d\&m conditions for~$P$ into
        upper bounds on $\rho_*(P)$, and the majority of them fall
        into two categories: those based on renewal theory, and those
        based on coupling.  There are also several different forms of
        d\&m in the literature, and they are all quite similar.  We
        will study two particular versions of d\&m - one used by
        \citet{baxendale2005renewal} in conjunction with renewal
        theory, and another used by \citet{rosenthal1995minorization}
        in conjunction with coupling.  \pcite{baxendale2005renewal}
        d\&m conditions take the following form:
	\begin{enumerate}
		\item [(A1)] There exist $\lambda \in [0,1)$, $K \in
		[1,\infty)$, $C \in \mathcal{B}$, and measurable $V:
		\X \to [1,\infty)$ such that
		\[
		PV(x) \leq \lambda V(x) 1_{\X \setminus C}(x) + K1_C(x)
		\]
		for each $x \in \X$. \label{a1}
		\item [(A2)] There exist $\varepsilon \in (0,1]$ and a probability
		measure $\nu: \mathcal{B} \to [0,1]$ such that
		\[
		P(x,A) \geq \varepsilon \nu(A) \,
		\]
		for each $x \in C$ and $A \in \mathcal{B}$. \label{a2}
		\item[(A3)] There exists $\beta
		\in (0,1]$ such that $\nu(C) \geq
		\beta$. \label{a3} 
	\end{enumerate}
	
	\noindent Here, (A\hyperref[a1]{1}) and (A\hyperref[a2]{2})
        are called the drift condition and the minorization condition,
        respectively.  (A\hyperref[a3]{3}) is called the strong
        aperiodicity condition.  The function $V$ is the drift
        function, and~$C$ is referred to as a small set.  We call
        (A\hyperref[a1]{1}) and (A\hyperref[a2]{2}) ``single-step''
        drift and minorization because they only involve the one-step
        transition kernel~$P$.  An explicit upper bound on $\rho_*(P)$
        can be derived using (A\hyperref[a1]{1})-(A\hyperref[a3]{3}),
        and, typically, the bound is only a function of $(\lambda, K,
        \varepsilon,\beta)$.  Here's an example.
	
	\begin{theorem}{[\cite{baxendale2005renewal}, Theorem 1.3]}
		\label{thm:baxbound}
		Let~$P$ be a Mtk on $(\X,\mathcal{B})$.  Suppose that
		(A\hyperref[a1]{1})-(A\hyperref[a3]{3}) hold.
		Then~$P$ admits a unique stationary distribution, $\Pi$, and
		$\rho_*(P) < 1$.  Suppose further that~$P$ satisfies
		the following conditions:
		\begin{enumerate}
			\item[(S1)]  The chain is
			reversible, i.e., for each $f,g \in
			L^2(\Pi)$,
			\[
			\int_{\X} g(x) Pf(x) \, \Pi(\df x) = \int_{\X}
			f(x) Pg(x) \, \Pi(\df x) \,.
			\] \label{s1}
			\item[(S2)]  The chain is
			non-negative definite, i.e., for each $f \in
			L^2(\Pi)$,
			\[
			\int_{\X} f(x) Pf(x) \, \Pi(\df x) \geq 0 \,.
			\] \label{s2}
		\end{enumerate}
		Then
		\begin{equation} \label{eq:baxbound}
		\rho_*(P) \leq \max \big \{ \lambda,
		(1-\varepsilon)^{1/\alpha_*} \big \} 1_{\varepsilon < 1} +
		\lambda 1_{\varepsilon = 1} \,,
		\end{equation}
		where
		\[
		\alpha_* = \frac{\log \big[
			(K-\varepsilon)/(1-\varepsilon) \big] + \log
			\lambda^{-1}}{\log \lambda^{-1}} \,.
		\]
		(When $\varepsilon \in (0,1)$ and $\lambda = 0$,~$\alpha_*$ is
		interpreted as~$1$.)
	\end{theorem}
	
	\begin{remark}
		\citet{baxendale2005renewal} also provides results for
                chains that are reversible, but not non-negative
                definite, and also for chains that are neither
                reversible nor non-negative definite.  However, those
                bounds are much more complex than \eqref{eq:baxbound}.
                In particular,~$\beta$ does not
                enter~\eqref{eq:baxbound}, but does appear in the
                other bounds.
	\end{remark}
	\begin{remark}
		\citet{jerison2019} uses the theory of \textit{strong random times} to
		extend and improve upon some of \pcite{baxendale2005renewal} results,
		but is unable to improve upon the convergence rate bound given in
		\eqref{eq:baxbound}.
	\end{remark}

	Results of this type have proven to be extremely useful for
	establishing the geometric ergodicity of Markov chains on continuous
	state spaces.  That is, for establishing the \textit{qualitative}
	result that $\rho_*(P) < 1$.  However, upper bounds on $\rho_*(P)$
	that are constructed based on single-step drift and minorization, such
	as~\eqref{eq:baxbound}, have a reputation of being very conservative.
	In practice, it's not unusual for a bound of this type to be very
	close to unity, even when the chain being studied apparently converges
	quite rapidly.  The following example illustrates this situation.
	
	Let $\X = \mathbb{R}^{10}$, and let~$P$ be given by
	\[
	P(x, \df y) \propto \exp \left( -\frac{2}{3} \left\| y - \frac{x}{2}
	\right\|^2 \right) \df y \,,
	\]
	where $\|\cdot\|$ is the Euclidean norm.  So $P$ defines a Gaussian
	autoregressive chain on~$\X$.  The $10$-dimensional standard Gaussian
	distribution is the unique stationary distribution of the
	corresponding Markov chain.  The chain is reversible, non-negative
	definite, and it is well-known that $\rho_*(P) = 0.5$.  Let us now
	pretend that we do not know the true convergence rate, and consider
	using Theorem~\ref{thm:baxbound} to form an upper bound on
	$\rho_*(P)$.  In order to apply the theorem, we must establish
	(A\hyperref[a1]{1})-(A\hyperref[a3]{3}).  A standard drift function to
	use is $V(x) = \|x\|^2/k + 1$, where $k$ can be tuned.  We take
	$k=100$, since this appears to give good results.  The small set~$C$
	is usually chosen to be $\{x\in \mathbb{R}^{10}: V(x) \leq d\}$, where
	$d \geq 1$ can be optimized.  In this case, (A\hyperref[a1]{1}) holds
	whenever $d > 1+10/k = 1.1$.  In fact,
	\[
	PV(x) \leq \frac{10d+33}{40d} V(x) 1_{\X \setminus C}(x) +
	\frac{10 d+33}{40} 1_C(x) \,.
	\]
	Moreover, for each $d > 1.1$, $a > 0$, $x \in C$ and $A \in
	\mathcal{B}$,
	\begin{align*}
	&P(x,A) \\
	&= \int_{A} \frac{1}{(3\pi/2)^5} \exp \left(
	-\frac{2}{3} \left\| y - \frac{x}{2} \right\|^2 \right) \, \df
	y \\ & \geq \int_A \frac{1}{(3\pi/2)^5} \inf_{ \|x'\|^2/k + 1
		\leq d} \exp \left( -\frac{2}{3} \left\| y - \frac{x'}{2}
	\right\|^2 \right) \, \df y \\ & \geq \int_A
	\frac{1}{(3\pi/2)^5} \inf_{ \|x'\|^2/k + 1 \leq d} \exp
	\left\{ -\frac{2}{3} \left[ (1+a)\|y\|^2 + \left( 1 +
	\frac{1}{a} \right) \left\| \frac{x'}{2} \right\|^2 \right]
	\right\} \, \df y \\ & \geq \int_A \frac{1}{(3\pi/2)^5} \exp
	\left[ -\frac{2(a+1)}{3} \|y\|^2 - \frac{100(a+1)(d-1)}{6a}
	\right] \, \df y \\ &= \frac{1}{(a+1)^5} \exp \left[
	-\frac{100(a+1)(d-1)}{6a} \right] \nu(A) \,,
	\end{align*}
	where $\nu$ is the $10$-dimensional normal distribution with
	mean~$0$ and variance $3 I_{10} / [4(a+1)]$, with $I_{10}$
	being the $10 \times 10$ identity matrix.  Thus,
	(A\hyperref[a2]{2}) holds.  It's obvious that
	(A\hyperref[a3]{3}) holds as well.  Applying
	Theorem~\ref{thm:baxbound}, and optimizing over $(a,d)$,
	leads to the following result: $\rho_*(P) \le 0.99993$.  This bound is
	obviously extremely conservative - recall that $\rho_*(P) = 0.5$.
	This leads to the following intriguing question: (Q1) Is this terrible
	bound on $\rho_*(P)$ a result of the \textit{way} in which
	Theorem~\ref{thm:baxbound} was applied (i.e., a poorly chosen drift
	function, loose inequalities in the d\&m, etc.), or is it simply
	impossible to use Theorem~\ref{thm:baxbound} to produce a sharp bound
	in this example?  Another interesting (and more general) question is
	this: (Q2) Assuming that~\eqref{eq:baxbound} is not the best possible
	bound that can be constructed using
	(A\hyperref[a1]{1})-(A\hyperref[a3]{3}), how much better could we hope
	to do?

	The two questions posed in the previous paragraph can be
        answered using results developed in
        Section~\ref{sec:limitation}, which is where we introduce a
        general framework for evaluating the effectiveness of a set of
        d\&m conditions for bounding convergence rates.  Our answer to
        (Q2) is that Baxendale's bound is actually quite tight; see
        Remark~\ref{rem:tight} in Subsection~\ref{ssec:paraoptima}.
        We prove this by constructing a reversible, non-negative definite Markov chain that satisfies
        (A\hyperref[a1]{1})-(A\hyperref[a3]{3}), and has a convergence
        rate that is only slightly smaller than the right-hand side of
        \eqref{eq:baxbound}.  We answer (Q1) in
        Subsection~\ref{sssec:gaussian1} by proving that it is \textit{impossible} to use d\&m in the form of
        (A\hyperref[a1]{1})-(A\hyperref[a3]{3}) to get a tight upper
        bound on $\rho_*(P)$ in the Gaussian autoregressive example.
        In fact, we show that no upper bound derived from
        (A\hyperref[a1]{1})-(A\hyperref[a3]{3}) can be less than 0.922
        - a far cry from the true value of 0.5.

        Our answer to (Q1) is based on an application of
        Theorem~\ref{thm:modeloptima1} in
        Subsection~\ref{ssec:modeloptima}, which is one of our main
        results.  It provides a lower bound on the best possible upper
        bound that could be constructed using
        (A\hyperref[a1]{1})-(A\hyperref[a3]{3}).  In particular, if
        $P$ satisfies (A\hyperref[a1]{1})-(A\hyperref[a3]{3}) and
        $\Pi$ is its stationary measure, then the best upper bound on
        $\rho_*(P)$ that we could possibly get is bounded below by
        \begin{equation}
          \label{eq:lb_intro}
          \inf_{C \in \mathcal{B}: \Pi(C) > 0} \left[
            (1-\varepsilon_C)^{\lfloor 1/\Pi(C) \rfloor^{-1}} \right]
          \,,
        \end{equation}
        where $\lfloor \cdot \rfloor$ returns the largest integer that
        does not exceed its argument, and
	\begin{equation} 
          \label{eq:varepsilonC}
	\begin{aligned}
	\varepsilon_C = &\sup\{ \varepsilon \in (0,1]: \\
	& \mbox{
		(A\hyperref[a2]{2}) holds for } P \mbox{ and } C
	\mbox{ with } \varepsilon \mbox{ and some probability
		measure } \nu \} \,,
	\end{aligned}
	\end{equation}
	where the right-hand side of \eqref{eq:varepsilonC} is
        interpreted as~$0$ if (A\hyperref[a2]{2}) doesn't hold on~$C$.
        An inspection of \eqref{eq:lb_intro} reveals a potentially
        serious limitation on convergence rate bounds based on
        (A\hyperref[a1]{1})-(A\hyperref[a3]{3}).  First note that, if
        \eqref{eq:lb_intro} is near 1, then it is impossible to get a
        sharp bound on $\rho_*(P)$ if $\rho_*(P)$ is not close to 1.
        Moreover, \eqref{eq:lb_intro} will be far from 1 only if we
        can find a set $C$ such that $\varepsilon_C$ and $\Pi(C)$ are
        both large.  Unfortunately, $\varepsilon_C$ tends to decrease
        as $C$ gets larger, while $\Pi(C)$ tends to increase as $C$
        grows.  Our examples reveal that, when dimension is high and
        $\Pi$ tends to ``spread out,'' the required ``Goldilocks'' $C$
        may not exist, even when $\rho_*(P)$ is not close to~$1$.

	At this point, we should make clear that, in theory, the
        problems with d\&m that are described above can be
        circumvented by moving from single-step d\&m to
        \textit{multi-step} d\&m.  Indeed, it is well known that, if
        one can establish d\&m conditions based on multi-step
        transition kernels, then the resultant convergence rate bounds
        can actually be quite well-behaved, even in high-dimensional
        settings where the single-step bounds fail completely.
        (See \citepage{qin2018wasserstein}{Appendix~A} for an example
        involving the Gaussian autoregressive process described
        above.)  Unfortunately, in practical situations where the
        transition law is highly complex (such as in MCMC), developing
        d\&m conditions with multi-step Mtks is usually impossible.
        For this reason, multi-step d\&m is seldom used in practice.
        For a more detailed discussion on this issue,
        see, e.g., Section~2 of \cite{qin2018wasserstein}.

	The optimality of convergence rate bounds based on d\&m has
        been touched on in previous work
        \citep{meyn1994computable,lund1996geometric,roberts2000rates,baxendale2005renewal,jerison2016drift,qin2018wasserstein}.
        For instance, \citepage{meyn1994computable}{page 986} compared
        their convergence rate bound with existing results for a
        certain class of chains, and concluded that their bound
        ``cannot be expected to be tight.''  In fact,
        \citet{baxendale2005renewal} developed a tighter version of
        their bound about a decade later.  However, to our knowledge,
        there has been no prior systematic study of the limitations of
        the d\&m methodology in general.
	
	The rest of this article is organized as follows.  In
        Section~\ref{sec:limitation}, we introduce our general
        framework for studying the limitations of a set of d\&m
        conditions for bounding convergence rates.  Within this
        framework, we analyze the sharpness of
        \pcite{baxendale2005renewal} bound, and, more generally, we
        consider optimal bounds based on
        (A\hyperref[a1]{1})-(A\hyperref[a3]{3}).  Our results are
        applied to the Gaussian autoregressive example, and also to
        the Metropolis adjusted Langevin algorithm (MALA) in a
        situation where the target is high-dimensional.  In
        Section~\ref{sec:coupling}, we study optimal bounds based on
        \pcite{rosenthal1995minorization} d\&m conditions, and the
        results are applied to the Gaussian autoregressive example,
        MALA, and two random walks on graphs.  Potential strategies
        for overcoming the limitations of single-step d\&m are
        discussed in Section~\ref{sec:beyondoptima}.  The Appendix
        contains several technical formulas and proofs.

	\section{Quantifying the Limitations of Bounds Based on (A\hyperref[a1]{1})-(A\hyperref[a3]{3})}
	\label{sec:limitation}
	
	\subsection{Parameter-specific optimal bound}
	\label{ssec:paraoptima}
	
	Consider a generic upper bound on the convergence rate that is based
	on (A\hyperref[a1]{1})-(A\hyperref[a3]{3}).  If this bound depends on
	(A\hyperref[a1]{1})-(A\hyperref[a3]{3}) only through the d\&m
	parameter, $(\lambda,K,\varepsilon,\beta)$, then we call it a
	\textit{simple} upper bound.  Just to be clear, a simple upper bound
	cannot use the drift function, $V(\cdot)$, the small set $C$, nor the
	minorization measure $\nu(\cdot)$.  As an example, the bound in
	Theorem~\ref{thm:baxbound} is simple.  Now fix a value of
	$(\lambda,K,\varepsilon,\beta)$.  
	Evidently, the smallest possible
	simple upper bound on the convergence rate of any Markov chain that
	satisfies (A\hyperref[a1]{1})-(A\hyperref[a3]{3}) with this value of
	the d\&m parameter is equal to the convergence rate of the
	\textit{slowest} Markov chain among \textit{all} the chains that
	satisfy (A\hyperref[a1]{1})-(A\hyperref[a3]{3}) with this particular
	value of $(\lambda,K,\varepsilon,\beta)$.  
	(To be more precise, it's the supremum of the convergence rates of these chains.)
	In this section, we will
	approximate this \textit{optimal} simple upper bound by
	constructing a class of slowly converging Markov chains that satisfy
	(A\hyperref[a1]{1})-(A\hyperref[a3]{3}).
	
	For each fixed value of $(\lambda,K,\varepsilon,\beta)$ in the set
	$T_0 := [0,1) \times [1,\infty) \times (0,1] \times (0,1]$, let
	$S^{(N)}_{\lambda,K,\varepsilon,\beta}$ denote the collection
	of reversible and non-negative definite Mtks that satisfy (A\hyperref[a1]{1})-(A\hyperref[a3]{3}) with that
	d\&m parameter.  We do not require these chains to have a common state space.
	As an example, consider
	$S^{(N)}_{0,1,1,1}$, which consists of all reversible and non-negative definite Mtks that satisfy the following:
	\begin{enumerate}
		\item There exist $C \in \mathcal{B}$ and measurable $V: \X \to
		[1,\infty)$ such that for each $x \in \X$,
		\begin{equation} \label{eq:drift-0}
		PV(x) \leq 1_C(x) \,.
		\end{equation}
		\item There exists a probability measure $\nu:
		\mathcal{B} \to [0,1]$ such that for each $x \in C$
		and $A \in \mathcal{B}$,
		\begin{equation} \label{eq:minor-0}
		P(x,A) \geq \nu(A) \,.
		\end{equation}
		\item  The following is satisfied:
		\begin{equation} \label{eq:strongaper-0}
		\nu(C) = 1\,.
		\end{equation}
	\end{enumerate}
	Since the function $V$ is bounded below by~$1$, \eqref{eq:drift-0} can
	hold only if $C = \X$.  Hence, when~\eqref{eq:drift-0} holds, so
	does~\eqref{eq:strongaper-0}, and~\eqref{eq:minor-0} holds if and only
	if $P(x,\cdot) = \nu(\cdot)$ for each $x \in \X$.  So, in fact,
	$S^{(N)}_{0,1,1,1}$ consists precisely of all Mtks that define a trivial
	(independent) Markov chain on some countably generated state space.

	\begin{remark}
		Strictly speaking, $S^{(N)}_{\lambda,K,\varepsilon,\beta}$ is too large to be considered as a proper set in an axiomatic sense.
		In principle, one can impose appropriate restrictions on $S^{(N)}_{\lambda,K,\varepsilon,\beta}$ to make it a set.
		For instance, one can assume that the collection of all the measurable spaces on which the Mtks in $S^{(N)}_{\lambda,K,\varepsilon,\beta}$ are defined is a set.
		We assume throughout that $S^{(N)}_{\lambda,K,\varepsilon,\beta}$ is indeed a set, and that it contains at least all reversible and non-negative definite Mtks whose state spaces are finite sets of integers that satisfy (A\hyperref[a1]{1})-(A\hyperref[a3]{3}) with d\&m parameter $(\lambda,K,\varepsilon,\beta)$.
		Analogous assumptions will be imposed on similar constructions in Section~\ref{ssec:coupling} and Appendix~\ref{app:nonrev}.
	\end{remark}

	For $(\lambda,K,\varepsilon,\beta) \in T_0$, $S^{(N)}_{\lambda,K,\varepsilon,\beta}$ is non-empty, since $S^{(N)}_{0,1,1,1} \subset S^{(N)}_{\lambda,K,\varepsilon,\beta}$.
	The parameter-specific optimal bound corresponding to d\&m parameter 
	$(\lambda,K,\varepsilon,\beta)$ for reversible and non-negative chains is defined as follows:
	\[
	\rho^{(N)}_{\mbox{\scriptsize{opt}}}(\lambda,K,\varepsilon,\beta) = \sup_{P
		\in S^{(N)}_{\lambda,K,\varepsilon,\beta}} \rho_*(P) \;.
	\]
	Consider the significance of
	$\rho^{(N)}_{\mbox{\scriptsize{opt}}}(\lambda,K,\varepsilon,\beta)$.  If $\rho(\lambda,K,\varepsilon,\beta)$ is a simple upper bound constructed from (A\hyperref[a1]{1})-(A\hyperref[a3]{3}) with d\&m parameter
	$(\lambda,K,\varepsilon,\beta)$ in conjunction with (S\hyperref[s1]{1}) and (S\hyperref[s2]{2}) (reversibility and non-negative definiteness), then it must be no smaller than $\rho^{(N)}_{\mbox{\scriptsize{opt}}}(\lambda,K,\varepsilon,\beta)$.
	If $\rho(\lambda,K,\varepsilon,\beta)$ is constructed without the assumption of reversibility and non-negative definiteness, it will be no better, and thus also lower bounded by $\rho^{(N)}_{\mbox{\scriptsize{opt}}}(\lambda,K,\varepsilon,\beta)$.
	(See Appendix~\ref{app:nonrev} for a detailed explanation.)
	Hence, for $P \in S^{(N)}_{\lambda,K,\varepsilon,\beta}$,
	if 
	there is another Mtk in $S^{(N)}_{\lambda,K,\varepsilon,\beta}$ that
	converge substantially slower than $P$, 
	then no simple upper bound based on this value of $(\lambda,K,\varepsilon,\beta)$ can
	provide a tight bound on $\rho_*(P)$.  
	Of course, $P$ may satisfy (A\hyperref[a1]{1})-(A\hyperref[a3]{3}) with many different values of $(\lambda,K,\varepsilon,\beta)$.
	We will confront this
	complication in the next subsection.

	While it is usually impossible to calculate parameter-specific optimal
	bounds exactly, it's easy to find lower bounds.  Fix
	$(\lambda,K,\varepsilon,\beta) \in T_0$, and suppose that
	$P_{\lambda,K,\varepsilon,\beta} \in S^{(N)}_{\lambda,K,\varepsilon,\beta}$.
	Then
	\[
	\rho^{(N)}_{\mbox{\scriptsize{opt}}}(\lambda,K,\varepsilon,\beta) \ge
	\rho_*(P_{\lambda,K,\varepsilon,\beta}) \;.
	\]
	Now, if for each $(\lambda,K,\varepsilon,\beta) \in T_0$, we can find
	a particularly slowly converging chain in
	$S^{(N)}_{\lambda,K,\varepsilon,\beta}$, then we can construct a good lower
	bound on $\rho^{(N)}_{\mbox{\scriptsize{opt}}}$.  Indeed, this is precisely
	how the next theorem is proven.

	\begin{theorem}
		\label{thm:paraoptima}
		For each $(\lambda,K,\varepsilon,\beta) \in T_0$, we have
		\begin{equation}
		\label{eq:paraoptima}
		\rho^{(N)}_{\mbox{\scriptsize{opt}}}(\lambda,K,\varepsilon,\beta)
		\geq \max\{ \lambda, (1-\varepsilon)^{1/\alpha} \}
		1_{\varepsilon < 1} + \lambda 1_{\varepsilon = 1} \,,
		\end{equation}
		where
		\[
		\alpha = \lfloor \alpha_* \rfloor = \left\lfloor \frac{\log
			[(K-\varepsilon) / (1-\varepsilon) ] + \log \lambda^{-1}}{\log
			\lambda^{-1}} \right\rfloor \,.
		\]
		(When $\varepsilon \in (0,1)$ and $\lambda = 0$,~$\alpha_*$ is
		interpreted as~$1$.)
	\end{theorem}
	
	
	\begin{proof}
		It is straightforward to verify that
		\[
		\max\{ \lambda, (1-\varepsilon)^{1/\alpha} \} 1_{\varepsilon < 1} +
		\lambda 1_{\varepsilon = 1} = \max \{ (1-\varepsilon)^{1/\alpha}
		1_{\varepsilon < 1}, \lambda \} \,,
		\]
		so it suffices to show that
		\begin{equation}
		\nonumber \rho^{(N)}_{\mbox{\scriptsize{opt}}} (\lambda, K,
		\varepsilon, \beta) \geq \max \{ (1-\varepsilon)^{1/\alpha}
		1_{\varepsilon < 1}, \lambda \} \,.
		\end{equation}
		We first prove that
		\begin{equation} 
		\label{eq:example-1}
		\rho^{(N)}_{\mbox{\scriptsize{opt}}} (\lambda, K, \varepsilon,
		\beta) \geq (1-\varepsilon)^{1/\alpha} 1_{\varepsilon < 1} \;.
		\end{equation}
		Indeed, for each value of $(\lambda, K, \varepsilon, \beta) \in T_0$
		such that $\varepsilon < 1$, we identify a $P_{\lambda, K,
			\varepsilon, \beta} \in S^{(N)}_{\lambda,K,\varepsilon,\beta}$ such
		that
		\begin{equation}
		\label{eq:example-case1-1}
		\rho_*(P_{\lambda, K, \varepsilon, \beta}) \geq
		(1-\varepsilon)^{1/\alpha} \,.
		\end{equation}
		Fix $\varepsilon < 1$, let $(\lambda, K, \beta) \in [0,1) \times
		[1,\infty) \times (0,1]$ be arbitrary, and note that $\alpha \ge 1$.
		Let $P_{\lambda, K, \varepsilon, \beta}$ be the Mtk of a Markov chain $\{X_m\}_{m=0}^{\infty}$ on the state space $\X = \{0,1, \dots, \alpha \}$ that adheres to the following rules:
		\begin{itemize}
			\item If $X_m = 0$, then $X_{m+1} = 0$.
			\item If $X_m = 1$, then with probability $\varepsilon$, $X_{m+1}
			= 0$, and with probability $1-\varepsilon$, $X_{m+1} = \alpha$.
			\item If $X_m \geq 2$, then $X_{m+1} = X_m-1$.
		\end{itemize}
		A Markov transition diagram for this chain is shown in
		Figure~\ref{fig:fig1}.
		(Since we do not require Mtks in $S^{(N)}_{\lambda,K,\varepsilon,\beta}$ to have a common state space, no generality is lost from focusing on a discrete state space.)

		\begin{figure}[ht]
			\centering
			\includegraphics[width=0.6\textheight]{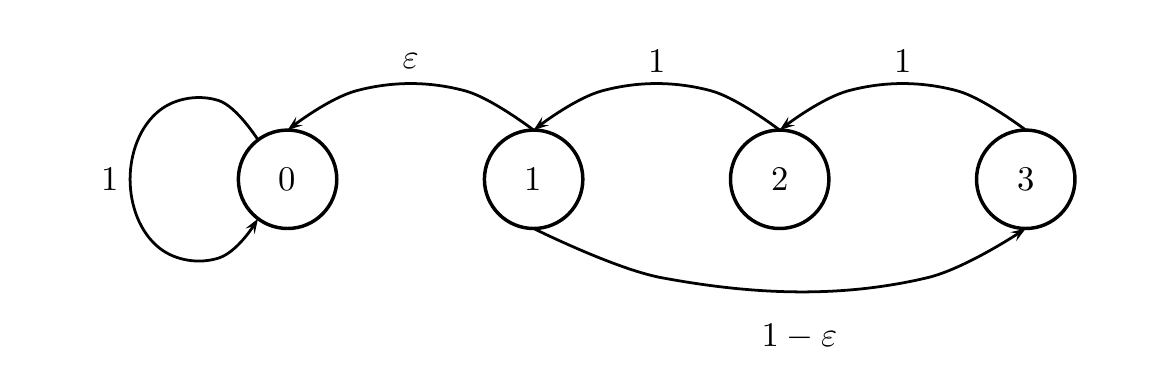}
			\caption{Markov transition diagram for the Markov chain
				$\{X_m\}_{m=0}^\infty$ when $\alpha=3$.}
			\label{fig:fig1}
		\end{figure}
		
		The chain admits a unique stationary distribution, which is precisely
		the point mass at~$0$, and it's straightforward to verify that the
		chain is reversible and non-negative definite.  
		(It is also aperiodic and
		$\psi$-irreducible, although it's \textit{not} irreducible in the
		classical sense that any two states communicate with each other.
		See \citepage{meyn2009markov}{Chapters 4\&5} for definitions.)
		
		We now verify that $P_{\lambda,K,\varepsilon,\beta}$ does satisfy
		(A\hyperref[a1]{1}), (A\hyperref[a2]{2}), and (A\hyperref[a3]{3}) with
		parameters $(\lambda, K, \varepsilon, \beta)$, i.e., $P_{\lambda, K,
			\varepsilon, \beta} \in S^{(N)}_{\lambda,K,\varepsilon,\beta}$.  To this end, let $C =
		\{0,1\}$, and define $V: \X \to [1,\infty)$ as follows.
		\begin{itemize}
			\item Let $V(0) = 1$, and let $V(\alpha) =
			(K-\varepsilon)/(1-\varepsilon)$.
			\item If $\alpha \geq 2$, let $V(x) = \lambda^{-x+1}$ for $x =
			1,2,\dots,\alpha-1$.
		\end{itemize}
		It's clear that (A\hyperref[a2]{2}) holds with $\nu$ being the point
		mass at~$0$.  
		(A\hyperref[a3]{3}) holds with $\beta = 1$, and thus with any~$\beta$.
		If $\alpha = 1$, then
		(A\hyperref[a1]{1}) obviously holds as well.  Assume
		that $\alpha \geq 2$.  Then $PV(0) = 1$, $PV(1) = K$,
		and thus, $PV(x) \leq K$ for each $x \in C$.  For $x =
		2,3,\dots,\alpha-1$ (if $\alpha \geq 3$), $PV(x) =
		V(x-1) = \lambda V(x)$.  Finally, noting that
		\[
		\alpha - 2 = \left\lfloor \frac{\log [(K-\varepsilon) / (1-\varepsilon) ] - \log \lambda^{-1}}{\log \lambda^{-1}} \right\rfloor \leq \frac{\log [(K-\varepsilon) / (1-\varepsilon) ] - \log \lambda^{-1}}{\log \lambda^{-1}} \,,
		\]
		we have
		\[
		\frac{PV(\alpha)}{V(\alpha)} = \frac{\lambda^{-\alpha + 2}}{(K-\varepsilon)/(1-\varepsilon)} \leq \lambda \,.
		\]
		Thus, (A\hyperref[a1]{1}) is satisfied.
		
		We now show that~\eqref{eq:example-case1-1} holds.
		When $\alpha = 1$,
		$d_{\mbox{\scriptsize{TV}}}(\delta_0 P_{\lambda, K,
			\varepsilon, \beta}^m, \delta_0) = 0$, and
		$d_{\mbox{\scriptsize{TV}}}(\delta_1 P_{\lambda, K,
			\varepsilon, \beta}^m, \delta_0) =
		(1-\varepsilon)^m$ for each~$m$.  Thus,~\eqref{eq:example-case1-1}
		holds.  For the remainder of this proof, assume that
		$\alpha \geq 2$.  It's easy to see that, for each
		positive integer~$m$ and $x \in \X$.
		\[
		d_{\mbox{\scriptsize{TV}}}(\delta_x P_{\lambda, K,
			\varepsilon, \beta}^m, \delta_0) = \mathbb{P}(X_m
		\neq 0|X_0 = x) \,.
		\]
		Suppose that $X_0 \neq 0$.  Each time the chain enters
		$\{1\}$, with probability~$\varepsilon$, it arrives
		at~$0$ in the next iteration, and stays there forever;
		with probability $1-\varepsilon$, it goes to~$\alpha$,
		and then takes exactly $\alpha-1$ steps to get back.
		Therefore, for any positive integer~$k$, the
		probability that the chain does not arrive at~$0$
		within $k\alpha$ iterations is $(1-\varepsilon)^k$.
		It follows that~\eqref{eq:example-case1-1} must hold,
		and~\eqref{eq:example-1} is satisfied.
		
		We now show that
		\begin{equation} \label{eq:example-2}
		\rho^{(N)}_{\mbox{\scriptsize{opt}}} (\lambda, K,
		\varepsilon, \beta) \geq \lambda \,,
		\end{equation}
		thereby completing the proof.  Let $(\lambda,
		\varepsilon, K, \beta) \in T_0$ be arbitrary.  If
		$\lambda = 0$, then~\eqref{eq:example-2} trivially
		holds.  Suppose that $\lambda > 0$, and let $\delta
		\in (0,\lambda)$.  Let $\tilde{P}_{\lambda, K,
			\varepsilon, \beta}$ be the Mtk
		of a Markov chain $\{\tilde{X}_m\}_{m=0}^{\infty}$ on the state space $\tilde{\X} = \{0,1\}$
		that adheres to the following rules:
		\begin{itemize}
			\item If $\tilde{X}_m = 0$, then $\tilde{X}_{m+1} = 0$.
			\item If $\tilde{X}_m = 1$, then with probability $\lambda -
			\delta$, $\tilde{X}_{m+1} = 1$, and with probability $ 1 - \lambda
			+ \delta$, $\tilde{X}_{m+1} = 0$.
		\end{itemize}
		As before, the unique invariant distribution of this chain is the
		point mass at~$0$, and it's easy to show that the chain is reversible
		and non-negative definite.  Let $C = \{0\}$, and define $V: \tilde{\X} \to
		[1,\infty)$ as follows: $V(0) = 1$, $V(1) =
		(1-\lambda+\delta)/\delta$.  It's easy to see that
		(A\hyperref[a1]{1}), (A\hyperref[a2]{2}), and (A\hyperref[a3]{3})
		all hold for $\tilde{P}_{\lambda, K, \varepsilon, \beta}$.  Finally,
		$\rho_*(\tilde{P}_{\lambda, K, \varepsilon, \beta}) = \lambda -
		\delta$, which implies that
		\[
		\rho^{(N)}_{\mbox{\scriptsize{opt}}} (\lambda, K, \varepsilon, \beta)
		\geq \lambda - \delta \,.
		\]
		Since $\delta \in (0,\lambda)$ is arbitrary,~\eqref{eq:example-2}
		holds, and the proof is complete.
	\end{proof}
	
	\begin{remark}
        \label{rem:tight}
		Combining Theorems~\ref{thm:baxbound}
                and~\ref{thm:paraoptima} yields the following
		\[
		\begin{aligned}
		\max\{ \lambda, (1-\varepsilon)^{1/\alpha} \}
                1_{\varepsilon < 1} + \lambda 1_{\varepsilon = 1} &\le
                \rho^{(N)}_{\mbox{\scriptsize{opt}}}(\lambda,K,\varepsilon,\beta)
                \\ & \le \max\{ \lambda, (1-\varepsilon)^{1/\alpha_*}
                \} 1_{\varepsilon < 1} + \lambda 1_{\varepsilon = 1}
                \;,
		\end{aligned}
		\]
		where $\alpha = \lfloor \alpha_* \rfloor$.  Note that
                the expressions on the right and left-hand sides are
                nearly identical.  Consequently, for reversible and
                non-negative definite chains,~\eqref{eq:baxbound} is
                close to optimal as a simple upper bound,
                and~\eqref{eq:paraoptima} is close to tight as a lower
                bound on
                $\rho^{(N)}_{\mbox{\scriptsize{opt}}}(\lambda, K,
                \varepsilon, \beta)$.  It's also worth noting that the
                parameter~$\beta$ does not enter the bounds
                in~\eqref{eq:baxbound} or~\eqref{eq:paraoptima}.  This
                implies that~$\beta$ is not important for chains that
                are reversible and non-negative definite, as far as
                optimal bounds are concerned.
	\end{remark}
	
	While it is certainly of interest to develop good lower bounds on the
	parameter-specific optimal bounds, such bounds do not provide us with
	much information about the effectiveness of drift and minorization for
	a given Markov chain.  Indeed, a single geometrically ergodic Markov
	chain would presumably satisfy (A\hyperref[a1]{1})-(A\hyperref[a3]{3})
	for many different values of $(\lambda, K, \varepsilon, \beta)$.  In
	the next subsection, we deal with this extra layer of complexity.
	
	\subsection{Chain-specific optimal bound}
	\label{ssec:modeloptima}
	
	Consider an Mtk~$P$ on some countably generated state space $(\X, \mathcal{B})$.  In this
	subsection, we investigate the best possible simple upper bound on
	$\rho_*(P)$ that can be obtained from
	(A\hyperref[a1]{1})-(A\hyperref[a3]{3}) and, if applicable, (S\hyperref[s1]{1}) and (S\hyperref[s2]{2}), after the d\&m parameter has been optimized.
	Denote this chain-specific optimal bound by $\rho_{\mbox{\scriptsize{opt}}}^*(P)$.
	In general, there isn't much hope of
	calculating $\rho_{\mbox{\scriptsize{opt}}}^*(P)$ exactly.  In what follows,
	we describe a framework to bound it from below.

	Define $T(P) \subset T_0$ as follows: $(\lambda, K, \varepsilon,
	\beta) \in T(P)$ if $P$ satisfies
	(A\hyperref[a1]{1})-(A\hyperref[a3]{3}) with this value of the d\&m
	parameter.  
	Let
	\[
	\rho_{\mbox{\scriptsize{opt}}}^{*(N)}(P) = \inf_{(\lambda, K, \varepsilon,
		\beta) \in T(P)} \rho^{(N)}_{\mbox{\scriptsize{opt}}}(\lambda, K,
	\varepsilon, \beta) \,.
	\]
	If $T(P) = \emptyset$, then
	$\rho_{\mbox{\scriptsize{opt}}}^{*(N)}(P)$ is set to be unity.
	$\rho_{\mbox{\scriptsize{opt}}}^{*(N)}(P)$ represents
	the best possible simple upper bound on $\rho_*(P)$ that can be
	constructed using (A\hyperref[a1]{1})-(A\hyperref[a3]{3}), assuming that (S\hyperref[s1]{1}) and (S\hyperref[s2]{2}) hold.  
	Hence, if~$P$ is reversible and non-negative definite, then $\rho_{\mbox{\scriptsize{opt}}}^*(P) = \rho_{\mbox{\scriptsize{opt}}}^{*(N)}(P)$.
	If~$P$ does not satisfy both (S\hyperref[s1]{1}) and (S\hyperref[s2]{2}), then any simple upper bound on $\rho_*(P)$ has to be constructed without the help of (S\hyperref[s1]{1}) and/or (S\hyperref[s2]{2}), and cannot be better than $\rho_{\mbox{\scriptsize{opt}}}^{*(N)}(P)$.
	In this case, $\rho_{\mbox{\scriptsize{opt}}}^{*(N)}(P)$ serves as a lower bound on $\rho_{\mbox{\scriptsize{opt}}}^*(P)$.
	(The exact formulas for $\rho_{\mbox{\scriptsize{opt}}}^*(P)$ when (S\hyperref[s1]{1}) and (S\hyperref[s2]{2}) do not both hold are given in Appendix~\ref{app:nonrev}.)

	Let $\hat{\rho}(P)$ be any (nontrivial) simple upper bound on
	$\rho_*(P)$ based on (A\hyperref[a1]{1})-(A\hyperref[a3]{3}) (and, if applicable, (S\hyperref[s1]{1}) and (S\hyperref[s2]{2})).  Then
	\[
	0 \leq \rho_*(P) \leq \rho_{\mbox{\scriptsize{opt}}}^*(P) \leq
	\hat{\rho}(P) \leq 1 \,.
	\]
	The effectiveness of d\&m for
	constructing an upper bound on $\rho_*(P)$ can be quantified by the
	gap between $\rho_*(P)$ and $\rho_{\mbox{\scriptsize{opt}}}^*(P)$.  A
	large gap means that there does not exist a realization of
	(A\hyperref[a1]{1})-(A\hyperref[a3]{3}) that yields a sharp simple
	upper bound on the chain's true convergence rate.  In particular, if
	$\rho_{\mbox{\scriptsize{opt}}}^*(P) \approx 1$, then it's impossible
	to use simple bounds based on (A\hyperref[a1]{1})-(A\hyperref[a3]{3})
	to show that the chain mixes rapidly, even if $\rho_*(P)$ is very
	small.
	
	In the
	previous section, we developed reasonably sharp lower bounds on
	$\rho^{(N)}_{\mbox{\scriptsize{opt}}}(\lambda, K, \varepsilon,
	\beta)$.  Thus, if we could identify $T(P)$, then it would, in
	principle, be straightforward to bound
	$\rho_{\mbox{\scriptsize{opt}}}^*(P)$ from below.  Unfortunately,
	identifying $T(P)$ requires finding all of the values of $(\lambda,
	K)$ for which (A\hyperref[a1]{1}) holds, which is impossible.  Indeed,
	in practice, the only drift conditions that can be established are
	those associated with simple drift functions that lend themselves to
	the analysis of the Markov chain corresponding to $P$.  Put simply,
	for a given $P$, there is a massive difference between the set of
	drift conditions that hold in theory, and the set of drift conditions
	that can actually be established in practice.  In what follows, we
	circumvent this difficulty by constructing a lower bound on
	$\rho^{(N)}_{\mbox{\scriptsize{opt}}}(\lambda, K, \varepsilon, \beta)$
	that does not depend on $(\lambda, K)$.  The construction pivots on a
	lower bound for the size of the small set.  A proof of the following
	result is provided in Appendix~\ref{app:pic-1}.
	
	\begin{theorem} \label{thm:pic-1}
		Suppose that~$P$ defines a $\psi$-irreducible,
		aperiodic Markov chain that satisfies
		(A\hyperref[a1]{1}) and (A\hyperref[a2]{2}) with
		parameter $(\lambda, K, \varepsilon) \in [0,1) \times
		[1,\infty) \times (0,1]$ and small set~$C$.
		Then~$P$ admits a unique stationary
		distribution~$\Pi$, and
		\[
		\Pi(C) \geq \frac{\log \lambda^{-1}}{\log K + \log \lambda^{-1}} \,.
		\]
		The right-hand side is interpreted as~$1$ if $\lambda
		= 0$.
	\end{theorem}

	Combining Theorems~\ref{thm:paraoptima} and~\ref{thm:pic-1}
	yields the following result.
	
	\begin{corollary}
		\label{cor:modeloptima}
		Let~$P$ be a Mtk that satisfies
		(A\hyperref[a1]{1})-(A\hyperref[a3]{3}) with d\&m parameter
		$(\lambda, K, \varepsilon, \beta) \in T_0$ and small set $C \in
		\mathcal{B}$.  Then~$P$ admits a unique stationary
		distribution~$\Pi$ such that $\Pi(C) > 0$, and
		\[
		\rho^{(N)}_{\mbox{\scriptsize{opt}}}(\lambda, K, \varepsilon, \beta)
		\geq (1-\varepsilon)^{\lfloor 1/\Pi(C) \rfloor^{-1}} \,.
		\]
	\end{corollary}

	\begin{proof}
		By Theorem~\ref{thm:baxbound},
		(A\hyperref[a1]{1})-(A\hyperref[a3]{3}) implies the existence of a unique stationary distribution~$\Pi$ as well as geometric ergodicity.
		Geometric ergodicity implies aperiodicity and $\psi$-irreducibility.
		By
		Theorem~\ref{thm:pic-1}, $\Pi(C) > 0$.  
		Therefore,
		$(1-\varepsilon)^{\lfloor 1/\Pi(C) \rfloor^{-1}} = 0$ when
		$\varepsilon = 1$.  It suffices to consider the case that
		$\varepsilon < 1$.  Fix $\varepsilon < 1$.  Since $K \geq 1$,
		$(K-\varepsilon)/(1-\varepsilon) \geq K$.  Hence,
		\[
		\alpha = \left\lfloor \frac{\log [(K-\varepsilon) / (1-\varepsilon) ]
			+ \log \lambda^{-1}}{\log \lambda^{-1}} \right\rfloor \geq
		\left\lfloor \frac{\log K + \log \lambda^{-1}}{\log \lambda^{-1}}
		\right\rfloor \,.
		\]
		By Theorem~\ref{thm:pic-1}, $1/\alpha \leq \lfloor 1/\Pi(C)
		\rfloor^{-1}$.  The result then follows from
		Theorem~\ref{thm:paraoptima}.
	\end{proof}
	\noindent Here is the main result of this section.
	\begin{theorem}
		\label{thm:modeloptima1}
		Let $P$ be a Mtk such that $T(P) \neq \emptyset$, and let~$\Pi$
		denote the stationary distribution.  Then
		\[
		\rho_{\mbox{\scriptsize{opt}}}^*(P) \geq \inf_{C \in \mathcal{B}:
			\Pi(C) > 0} \left[ (1-\varepsilon_C)^{\lfloor 1/\Pi(C) \rfloor^{-1}}
		\right] \,,
		\]
                where $\varepsilon_C$ is defined at
                \eqref{eq:varepsilonC}.
	\end{theorem}
	
	\begin{proof}
		For each $(\lambda, K, \varepsilon, \beta) \in T(P)$, there exists
		$C' \in \mathcal{B}$ such that (A\hyperref[a1]{1}) to
		(A\hyperref[a3]{3}) hold with parameter $(\lambda, K, \varepsilon,
		\beta)$ and small set~$C'$.  By Corollary~\ref{cor:modeloptima},
		$\Pi(C') > 0$, and
		\[
		\rho^{(N)}_{\mbox{\scriptsize{opt}}}(\lambda, K, \varepsilon, \beta)
		\geq (1-\varepsilon)^{\lfloor 1/\Pi(C') \rfloor^{-1}} \geq \inf_{C \in
			\mathcal{B}: \Pi(C) > 0} \left[ (1-\varepsilon_C)^{\lfloor 1/\Pi(C)
			\rfloor^{-1}} \right] \,.
		\]
		Recall that 
		\[
		\rho_{\mbox{\scriptsize{opt}}}^*(P) \ge \rho_{\mbox{\scriptsize{opt}}}^{*(N)}(P) := \inf_{(\lambda, K,
			\varepsilon, \beta) \in T(P)}
		\rho^{(N)}_{\mbox{\scriptsize{opt}}}(\lambda, K, \varepsilon, \beta) \,.
		\]
		The then result follows.
	\end{proof}
	
	Any simple upper bound on $\rho_*(P)$ based on
        (A\hyperref[a1]{1})-(A\hyperref[a3]{3}) cannot be smaller than
        $\rho_{\mbox{\scriptsize{opt}}}^*(P)$.  Of course, we would
        like $\rho_{\mbox{\scriptsize{opt}}}^*(P)$ to be as far away
        from~$1$ as possible.  Theorem~\ref{thm:modeloptima1} shows
        that $\rho_{\mbox{\scriptsize{opt}}}^*(P)$ is far away from
        unity only if one can find a set~$C$ such that $\varepsilon_C$
        and $\Pi(C)$ are both large.  But, as mentioned in the
        Introduction, $\varepsilon_C$ tends to decrease with the size
        of~$C$, while $\Pi(C)$ tends to increase with it.  As we shall
        see in the next subsection, when dimension is high and $\Pi$
        tends to ``spread out,'' it may be impossible to find a $C$
        that does the job, even when $\rho_*(P)$ is not close to~$1$.
	
	\subsection{Examples}
	\label{ssec:example-1}
	
	\subsubsection{Gaussian autoregressive chain}
	\label{sssec:gaussian1}
	
	Let us now consider a generalization of the Gaussian autoregressive
	example from the Introduction.  Let $\X = \mathbb{R}^n$, where $n$ is
	a positive integer, and define $P_n$ by
	\[
	P_n(x, \df y) \propto \exp \left( -\frac{2}{3} \left\| y -
	\frac{x}{2} \right\|^2 \right) \df y \,, \quad x \in
	\mathbb{R}^n \,.
	\]
	The corresponding Markov chain is reversible, non-negative definite,
	and has a unique stationary distribution~$\Pi_n$, which is the
	$n$-dimensional standard Gaussian distribution.  The chain is
	geometrically ergodic, and it's well-known that, regardless of~$n$,
	$\rho_*(P_n) = 0.5$.  In the Introduction, we established
	(A\hyperref[a1]{1})-(A\hyperref[a3]{3}) for $P_{10}$ using a quadratic
	drift function, and then applied Theorem~\ref{thm:baxbound}, which
	yielded an upper bound of 0.99993 for $\rho_*(P_{10})$.  We will use the results we have developed to demonstrate that the problem here is not a poorly
	chosen drift function, loose d\&m inequalities, etc.; it is, in fact,
	the inadequacy of the d\&m method itself.
	
	For a given positive integer~$n$, let $C_n \subset \mathbb{R}^n$ be
	measurable, and let $D_n$ be the diameter of $C_n$, i.e., $\sup \{
	\|x-y\|: x,y \in C_n \}$.  Let $B_{D_n/2}$ be the ball of
	diameter~$D_n$ that is centered at the origin, and let $f_n(\cdot)$ be
	the probability density function of a $\chi_n^2$ random variable,
	i.e., $f_n(x) \propto x^{(n-2)/2} e^{-x/2} 1_{(0,\infty)}(x)$.  Then
	\[
	\Pi_n(C_n) \leq \int_{B_{D_n/2}} \Pi_n(\df x) = \frac{1}{\alpha_n(D_n)}  \,,
	\]
	where 
	\[
	\alpha_n(D) = \left(\int_0^{D^2/4} f_n(x) \, \df x \right)^{-1} > 1
	\]
	for $D > 0$.
	
	Next, we construct an upper bound on $\varepsilon_{C_n}$.
	Let $p_n(x,y)$ be the density function of $P_n(x,\df y)$ with respect to the Lebesgue measure.  
	If (A\hyperref[a2]{2}) does not hold for $P_n$ and $C_n$, then $\varepsilon_{C_n} = 0$ by definition.
	Suppose that (A\hyperref[a2]{2}) holds for
	$P_n$ and $C_n$ with some $\varepsilon_n > 0$ and probability measure
	$\nu_n$, i.e., for each $x \in C_n$ and measurable $A \subset
	\mathbb{R}^n$,
	\begin{equation} \label{eq:gaussianminor}
	P_n(x,A) \geq \varepsilon_n \nu_n(A) \,.
	\end{equation}
	Then $\nu_n(\cdot)$ is absolutely continuous with respect to
	$P_n(x,\cdot)$ for each $x \in C_n$.  This implies that
	$\nu_n$ admits a measurable density.  It follows
	from~\eqref{eq:gaussianminor} that, for each $x \in C_n$ and
	almost every $y \in \mathbb{R}^n$,
	\[
	p_n(x,y) \geq \varepsilon_n \nu_n(\df y)/\df y \,.
	\]
	Hence, for each $x,x' \in C_n$,
	\[
	\int_{\mathbb{R}^n} \min \{ p_n(x,y), p_n(x',y) \} \, \df y
	\geq \int_{\mathbb{R}^n} \varepsilon_n \nu_n(\df y) \,.
	\]
	Therefore,
	\[
	\begin{aligned}
		\varepsilon_{C_n} 
		&\leq \inf_{x,x' \in C_n} \int_{\mathbb{R}^n} \min \{ p_n(x,y),
		p_n(x',y) \} \, \df y \\
		&= \inf_{x,x' \in C_n} \int_{\mathbb{R}^n} \left[ p_n(x,y) - |p_n(x,y) - p_n(x',y)| \right] \df y \,.
	\end{aligned}
	\]
	By standard results on the total variation distance between Gaussians \citep[see, e.g.,][page 5]{devroye2018total},
	\[
	\int_{\mathbb{R}_n} |p_n(x,y) - p_n(x',y)| \, \df y = 1 - 2\Phi\left( - \frac{\|x-x'\|}{2\sqrt{3}} \right),
	\]
	where $\Phi(\cdot)$ is the (cumulative) distribution function of the
	one-dimensional standard Gaussian distribution.  
	Thus,
	\begin{equation} \label{eq:gaussian-tv}
	\varepsilon_{C_n} \leq 2 \Phi\left(-\frac{D_n}{2\sqrt{3}} \right) \,.
	\end{equation}
	Therefore,
	Theorem~\ref{thm:modeloptima1} yields the following:
	\begin{equation} \label{ine:gaussianoptima}
	\rho_{\mbox{\scriptsize{opt}}}^*(P_n) \geq \inf_{D > 0} \left[ 1 -
	2\Phi\left( -\frac{D}{2\sqrt{3}} \right) \right]^{1/\lfloor
		\alpha_n(D) \rfloor} =: \rho_n^* \,.
	\end{equation}
	No simple upper bound on $\rho_*(P_n)$ based on
	(A\hyperref[a1]{1})-(A\hyperref[a3]{3}) can be less than $\rho_n^*$
	(even if it exploits the reversibility and non-negative definiteness
	of $P_n$).  For comparison with the analysis in the Introduction, we
	note that $\rho_{10}^* \approx 0.922$, which is much larger than the
	true convergence rate, $\rho_*(P_{10}) = 0.5$.  Thus, a simple bound
	based on (A\hyperref[a1]{1})-(A\hyperref[a3]{3}) cannot produce a
	tight bound on $\rho_*(P_{10})$.  Unfortunately, things only get worse
	as $n \rightarrow \infty$.  A proof of the following result can be
	found in Appendix~\ref{app:gaussian}.
	
	\begin{proposition}
		\label{pro:gaussian}
		$\rho_{\mbox{\scriptsize{opt}}}^*(P_n) \to 1$ as $n \to \infty$.
	\end{proposition}
	
	Proposition~\ref{pro:gaussian} shows that
	(A\hyperref[a1]{1})-(A\hyperref[a3]{3}) are completely inadequate for
	obtaining sharp upper bounds on $\rho_*(P_n)$ when~$n$ is large.  No
	matter what drift function and small set are used, and regardless of
	how the (simple) convergence bound is formed, this inadequacy
	persists.
	
	\subsubsection{Metropolis adjusted Langevin algorithm}
	\label{sssec:mala}
	
	Again, let $\X = \mathbb{R}^n$ for some positive integer~$n$.
	Let $\pi_n: \X \to [0,\infty)$ be a differentiable probability
	density function, and let $f_n(x) = -\log \pi_n(x) ,\, x \in
	\X$.  Denote the corresponding probability measure by
	$\Pi_n$.  The Metropolis adjusted Langevin algorithm (MALA)
	is an MCMC algorithm that can be used to draw random vectors that
	are approximately distributed as $\Pi_n$.  It is carried out
	by simulating a Markov chain $\{X_m\}_{m=0}^{\infty}$ with
	the following two-step transition rule.
	\begin{enumerate}
		\item {\it Proposal Step.} Given $X_m = x \in \X$, draw $y \in \X$ from the normal distribution $\mbox{N}_n(x - h_n \nabla f_n(x), 2h_n I_n)$, where $h_n > 0$ is the step size.
		\item {\it Metropolis Step.} Let
		\[
		a_n(x,y) = \min \left\{ 1, \frac{\pi_n(y) \exp [ -\|x-y+h_n \nabla f_n(y) \|^2/(4h_n) ]  }{\pi_n(x) \exp[ -\| y - x + h_n \nabla f_n(x) \|^2 /(4h_n) ] } \right\} \,.
		\]
		With probability $a_n(x,y)$, accept the proposal, and set $X_{m+1} = y$; with probability $1-a_n(x,y)$, reject the proposal, and set $X_{m+1} = x$.
	\end{enumerate}
	$\{X_m\}$ is reversible with respect to $\Pi_n$, but it's unknown if the chain is non-negative definite.
	Its transition kernel is given by
	\begin{equation} \label{eq:malakernel}
	\begin{aligned}
	P_n(x,A) =& \int_{\mathbb{R}^n} (4\pi h_n)^{-n/2} \exp \left( -\frac{1}{4h_n} \|y-x+h_n\nabla f_n(x)\|^2 \right) \times \\
	&\qquad\qquad \{ a_n(x,y)1_A(y) + [1-a_n(x,y)] 1_A(x) \} \df y \,.
	\end{aligned}
	\end{equation}
	We will study whether one can find a sharp upper bound on $\rho_*(P_n)$ via a drift and minorization argument based on (A\hyperref[a1]{1})-(A\hyperref[a3]{3}), particularly for large~$n$.
	
	The magnitude of the step size~$h_n$ plays an important role
	in the convergence of MALA.  \citepage{roberts1998optimal}{Section~2} argued
	that, if $\pi_n$ corresponds to independent and identically
	distributed random components, and the chain starts from
	stationarity, then one should set~$h_n$ to be of order
	$n^{-1/3}$.  However, when the chain does not start from
	$\Pi_n$, a step size of order $n^{-1/2}$ is more appropriate
	\citep[][Section~5]{christensen2005scaling}.  Step-sizes of similar order
	are also recommended in more recent studies \citep[see,
	e.g.,][]{dwivedi2018log}.
	
	While no concrete results have been established regarding the
	behavior of $\rho_*(P_n)$ as $n \to \infty$, recent results
	suggest that when~$h_n$ is chosen appropriately, $\rho_*(P_n)$
	does not tend to~$1$ rapidly as $n \to \infty$ \citep[see,
	e.g.,][Section~3]{dwivedi2018log}.  That is, it appears to be the
	case that MALA converges reasonably fast in high-dimensional
	settings.  Indeed, without the Metropolis step, the algorithm
	is just an Euler discretization of the Langevin diffusion,
	which is a stochastic process defined by the stochastic
	differential equation
	\[
	\df L_{n,t} = -\nabla f_n(L_{n,t}) \,\df t + \sqrt{2} \,\df W_{n,t} \,,
	\]
	where $\{W_{n,t}\}_{t \geq 0}$ is the standard Brownian motion on $\X =
	\mathbb{R}^n$.  Suppose that $f_n$ is strongly convex with
	parameter $\ell > 0$, i.e., for each $x,y \in \X$,
	\[
	f_n(x) - f_n(y) - \nabla f_n(y)^T (x-y) \geq \frac{\ell}{2} \|x-y\|^2 \,.
	\]
	Then under regularity conditions, the total variation distance
	between the distribution of $L_{n,t}$ and $\Pi_n$ is bounded
	above by a function of~$t$ which decays at a geometric rate of
	$e^{-\ell/2}$ \citep[see, e.g.,][Lemma~1]{dalalyan2017theoretical}.
	It seems reasonable to believe that, when~$h_n$ is small and
	the Metropolis step rarely results in a rejection,
	$\rho_*(P_n)^{1/h_n}$ is comparable to $e^{-\ell/2}$.
	When~$\ell$ is independent of~$n$, and~$h_n$ is of order
	$n^{-\gamma}$ for some constant $\gamma > 0$, this would
	suggest that $ \limsup_{n \to \infty} \rho_*(P_n)^{n^{\gamma}}
	\in [0,1) $.  If this is true, then
	\begin{equation} \nonumber
	\liminf_{n \to \infty} n^{\gamma} [1-\rho_*(P_n)] > 0 \,.
	\end{equation}
	That is, as $n \to \infty$, $\rho_*(P_n)$ goes to~$1$ at a
	polynomial (or slower) rate.  This is, of course, just
	conjecture.  One possible approach to proving this assertion
	is to establish, for each~$n$, a set of single-step drift and
	minorization conditions, such as (A\hyperref[a1]{1})-(A\hyperref[a3]{3}), and then to use these conditions to
	construct a quantitative upper bound on $\rho_*(P_n)$.  The
	assertion is proved if this upper bound converges to~$1$ at a
	polynomial rate as $n \to \infty$.  In what follows, we show
	that this argument is unlikely to work.

	For simplicity, we consider the case where $\pi_n$ satisfies
	the following conditions.
	\begin{enumerate}
		\item [(H1)] $\pi_n$ corresponds to independent
		identically distributed random variables. \label{h1}
		That is, there exists a probability density function
		$g: \mathbb{R} \to [0,\infty)$ independent of~$n$
		such that, for each $n \geq 1$ and $x = (x_1,
		\dots, x_n) \in \mathbb{R}^n$, $\pi_n(x) =
		\prod_{i=1}^{n} g(x_i)$.  Moreover, $\sup_{x \in
			\mathbb{R}} g(x) < \infty$.
		\item [(H2)] 
		There exists a constant~$M$ such that, for each $x_1,y_1 \in \mathbb{R}$,
		\[
		\left| -\frac{\df \log g(x_1)}{\df x_1} + \frac{\df \log g(y_1)}{\df y_1} \right| \leq M |x_1 - y_1| \,.
		\] \label{h2}
	\end{enumerate}
	\vspace{-.5cm}
	(H\hyperref[h1]{1}) and (H\hyperref[h2]{2}) imply that $f_n$ is $M$-smooth, i.e., for each $x,y \in \X$,
	\[
	\|\nabla f_n(x) - \nabla f_n(y) \| \leq M\|x-y\| \,.
	\]
	
	The following result is proved in Appendix~\ref{app:mala}.
	\begin{proposition} \label{pro:mala}
		Suppose that (H\hyperref[h1]{1}) and
		(H\hyperref[h2]{2}) are satisfied, and that $h_n <
		n^{-\gamma}$ for some $\gamma > 0$.  Then, for any
		$\gamma' > 0$,
		\[
		\lim_{n \to \infty} n^{\gamma'}
		[1-\rho_{\mbox{\scriptsize{opt}}}^*(P_n)]
		= 0 \,.
		\]
	\end{proposition}
	
	Proposition~\ref{pro:mala} shows that, under
	(H\hyperref[h1]{1}) and (H\hyperref[h2]{2}), it is not
	possible to construct a simple bound based on
	(A\hyperref[a1]{1})-(A\hyperref[a3]{3}) (and perhaps
	reversibility) that tends to~$1$ at a polynomial rate.
	
	\section{\pcite{rosenthal1995minorization} Drift and Minorization Conditions}
	\label{sec:coupling}
	
	In this section, we develop analogues of the results in
        Section~\ref{sec:limitation} for convergence rate bounds based
        on \pcite{rosenthal1995minorization} d\&m conditions, which
        take the following form:
	\begin{enumerate}
		\item[(B1)] There exist $\eta \in [0,1)$, $L \in
		[0,\infty)$, and measurable $V: \X \to [0,\infty)$
		such that
		\[
		PV(x) \leq \eta V(x) + L 
		\]
		for each $x \in \X$. \label{b1}
		\item[(B2)] There exist $\varepsilon \in (0,1]$, $d >
		0$, and a probability measure $\nu: \mathcal{B}
		\to [0,1]$ such that
		\[
		P(x, A) \geq \varepsilon \nu(A)
		\]
		for each $x \in C$ and $A \in \mathcal{B}$, where $C =
		\{x \in \X: V(x) \leq d \}$. \label{b2}
		\item[(B3)] The following holds: $d >
		2L/(1-\eta)$. \label{b3}
	\end{enumerate}
	
	The main difference between
        (B\hyperref[b1]{1})-(B\hyperref[b3]{3}) and
        (A\hyperref[a1]{1})-(A\hyperref[a3]{3}) is that the former
        puts a restriction on the structure of $C$.  Indeed, in
        (B\hyperref[b2]{2}),~$C$ is assumed to be a level set of the
        drift function.  Assume that
        (B\hyperref[b1]{1})-(B\hyperref[b3]{3}) hold.  The size of~$C$
        is controlled by the parameter~$d$, which is bounded below by
        $2L/(1-\eta)$.  The lower bound on~$d$ originates from the
        coupling argument, and it has some interesting implications.
        One of these is that $\Pi(C) > 1/2$
        \citep[][Proposition~2.16]{jerison2016drift}, which, as we
        shall see, turns out to be very important.  The following
        result provides a recipe for converting
        (B\hyperref[b1]{1})-(B\hyperref[b3]{3}) into a convergence
        rate bound.
	
	\begin{theorem}{[\cite{rosenthal1995minorization}, Theorem 12]}
		\label{thm:rosenbound}
		Let~$P$ be a Mtk on $(\X, \mathcal{B})$.  Suppose that
                (B\hyperref[b1]{1})-(B\hyperref[b3]{3}) hold.
                Then~$P$ admits a unique stationary distribution,
                $\Pi$, and
		\begin{equation}
		\label{eq:rosbound}
		\rho_*(P) \leq (1-\varepsilon)^{1/\alpha_{**}}
                1_{\varepsilon<1} + \tilde{\lambda} 1_{\varepsilon =
                  1} \,,
		\end{equation}
		where
		\[
		\alpha_{**} = \frac{\log [\tilde{K}/(1-\varepsilon)] +
                  \log \tilde{\lambda}^{-1} }{\log
                  \tilde{\lambda}^{-1}} \,,
		\]
		$\tilde{\lambda} = (1+2L+\eta d)/(1+d)$ and $\tilde{K}
                = 1 + 2\eta d + 2L$.
	\end{theorem}
	
	\begin{remark}
		In \cite{rosenthal1995minorization}, the upper bound
                on $\rho_*(P)$ depends upon an additional free
                parameter.  The bound presented here is the result of
                optimizing with respect to the free parameter.
	\end{remark}
	
	\begin{remark}
		In general, coupling arguments often produce
                convergence bounds that are simpler than those based
                on renewal theory.  While \eqref{eq:rosbound} is
                actually a bit more complex than~\eqref{eq:baxbound},
                note that~\eqref{eq:baxbound} is valid only when the
                Markov chain is reversible and non-negative definite,
                whereas \eqref{eq:rosbound} holds without this extra
                regularity.  Moreover, the bounds that
                \citet{baxendale2005renewal} developed for chains that
                do not satisfy the extra regularity are much more
                complex than~\eqref{eq:rosbound}.
	\end{remark}
	
	\begin{remark}
		\citepage{rosenthal1995minorization}{Theorem~5} also
                provides a version of Theorem~\ref{thm:rosenbound}
                based on multi-step minorization, that is,
                (B\hyperref[b2]{2}) with~$P$ replaced by $P^m$ for $m
                \geq 1$.
	\end{remark}

	\subsection{Optimal bounds} \label{ssec:coupling}
	
	We will mimic what was done in Section~\ref{ssec:paraoptima} with
	(B\hyperref[b1]{1})-(B\hyperref[b3]{3}) in place of
	(A\hyperref[a1]{1})-(A\hyperref[a3]{3}).  
	In this context, a ``simple'' upper bound
	based on (B\hyperref[b1]{1})-(B\hyperref[b3]{3}) is one that depends
	on the d\&m parameter $(\eta, L, \varepsilon, d)$, but not on $V$ or
	$\nu$. 
	For fixed $(\eta, L,
	\varepsilon, d) \in \tilde{T}_0 := [0,1) \times [0,\infty) \times
	(0,1] \times [0,\infty)$, define $S_{\eta,L,\varepsilon,d}$ to be
	the collection of Mtks that satisfy
	(B\hyperref[b1]{1})-(B\hyperref[b3]{3}) with d\&m parameter $(\eta,
	L, \varepsilon, d)$.   
	Let
	\[
	\rho_{\mbox{\scriptsize{opt}}}(\eta,L,\varepsilon,d) = \sup_{P \in
		S_{\eta,L,\varepsilon,d}} \rho_*(P) \;.
	\]
	If $S_{\eta,L,\varepsilon,d} = \emptyset$, $\rho_{\mbox{\scriptsize{opt}}}(\eta,L,\varepsilon,d)$ is set to be~$0$.
	$\rho_{\mbox{\scriptsize{opt}}}(\eta,L,\varepsilon,d)$ is the smallest simple upper bound that can be constructed based on (B\hyperref[b1]{1})-(B\hyperref[b3]{3}) with the given d\&m parameter value.
	
	\begin{proposition}
		\label{pro:paraoptima-B}
		For each $(\eta, L, \varepsilon, d) \in \tilde{T}_0$ such that $d > 2L/(1-\eta)$, we have
		\[
		\rho_{\mbox{\scriptsize{opt}}}(\eta, L, \varepsilon, d)
		\geq 1-\varepsilon \;.
		\]
	\end{proposition}
	
	\begin{proof}
		Let $(\eta, L, \varepsilon, d)$ be fixed, and let $P_{\eta, L,
			\varepsilon, d}$ be the Mtk of a Markov chain on the state space
		$\{0,1\}$ that adheres to the following rules:
		\begin{itemize}
			\item If $X_m = 0$, then $X_{m+1} = 0$.
			\item If $X_m = 1$, then $X_{m+1} = 0$ with probability
			$\varepsilon$, and $X_{m+1} = 1$ otherwise.
		\end{itemize}
		Let $V(0) = V(1) = 0$.  Then (B\hyperref[b1]{1})-(B\hyperref[b3]{3})
		all hold (with $C=\X$).  It follows that $P_{\eta, L, \varepsilon, d}
		\in S_{\eta,L,\varepsilon,d}$, and
		\[
		\rho_{\mbox{\scriptsize{opt}}}(\eta, L, \varepsilon, d ) \geq
		\rho_*(P_{\eta, L, \varepsilon, d}) = 1-\varepsilon \,.
		\]
	\end{proof}
	
	The lower bound on $\rho_{\mbox{\scriptsize{opt}}}$ given in
	Proposition~\ref{pro:paraoptima-B} is obviously very crude.
	Nevertheless, as we shall see, it's enough to produce several
	interesting results.  We now turn our attention to chain-specific
	optimal bounds.
	
	For a Mtk~$P$ on $(\X, \mathcal{B})$, define $\tilde{T}(P) \subset \tilde{T}_0$ as follows: $(\eta, L,
	\varepsilon, d) \in \tilde{T}(P)$ if $P$ satisfies
	(B\hyperref[b1]{1})-(B\hyperref[b3]{3}) with this value of the d\&m
	parameter.  Define the chain-specific optimal upper bound on $\rho_*(P)$
	as
	\[
	\rho_{\mbox{\scriptsize{opt}}}^\dagger(P) = \inf_{(\eta, L,
		\varepsilon, d) \in \tilde{T}(P)} \rho_{\mbox{\scriptsize{opt}}}(\eta, L,
	\varepsilon, d) \,.
	\]
	As usual, if $\tilde{T}(P) = \emptyset$, then $\rho_{\mbox{\scriptsize{opt}}}^\dagger(P)$ is set to be~$1$.
	
	To get a handle on $\rho_{\mbox{\scriptsize{opt}}}^\dagger(P)$, we study the size of the small set~$C$.
	The following result can be found in Chapter~2 of \cite{jerison2016drift}.
	\begin{proposition} \label{pro:jerison} \citep[][]{jerison2016drift}
		Suppose that (B\hyperref[b1]{1})-(B\hyperref[b3]{3}) hold.
		Then $\Pi(C) > 1/2$.
	\end{proposition}
	\begin{proof}
		Since $\Pi$ is stationary, it follows from (B\hyperref[b1]{1}) that
		\[
		\Pi V :=  \int_{\X} V(x) \Pi(\df x) \leq \eta \Pi V + L \,,
		\]
		which implies that $\Pi V \leq L/(1-\eta)$ \citep[Proposition
		4.24]{hairer2006ergodic}.  On the other hand, since $V(x) >
		d$ for $x \in \X \setminus C$, $\Pi V \geq d(1-\Pi(C))$.
		Combining the upper and lower bounds on $\Pi V$ yields
		\[
		1 - \Pi(C) \leq \frac{L}{d(1-\eta)} \,.
		\]
		Hence, by (B\hyperref[b3]{3}), $\Pi(C) > 1/2$.
	\end{proof}
	\begin{remark} \label{rem:jerison}
		As argued in \cite{jerison2016drift}, the bound $\Pi(C)>1/2$ holds much more generally.
		Indeed, in coupling arguments, drift conditions like (B\hyperref[b1]{1}) are usually used in conjunction with regularity conditions like (B\hyperref[b3]{3}) to establish bivariate drift conditions of the following form.
		\begin{enumerate}
			\item [(C1)] There
			exist $\lambda' < 1$, $K' \in [1,\infty)$, $C' \in \mathcal{B}$, and measurable function $V_1: \X \to
			[1/2,\infty)$ such that
			\begin{equation} \nonumber
			PV_1(x) + P V_1(y) \leq \lambda' [V_1(x) + V_1(y)]
			1_{(\X \times \X) \setminus (C' \times C')}(x,y) + K'
			1_{C'\times C'}(x,y)
			\end{equation}
			for each $x,y \in \X$. \label{c1}
		\end{enumerate}
		This type of bivariate drift condition is then used to bound the time it takes for two coupled copies of the Markov chain defined by~$P$ to enter the set~$C'$ simultaneously.
		For instance, to prove Theorem~\ref{thm:rosenbound}, \citet{rosenthal1995minorization} derived (C\hyperref[c1]{1}) for $V_1(x) = V(x) + 1/2$ and $C' = C$.
		(C\hyperref[c1]{1}) is also crucial for deriving d\&m-based bounds in \cite{roberts1999bounds} and \cite{roberts2004general}.
		It turns out that, whenever (C\hyperref[c1]{1}) holds, one must have $\Pi(C') > 1/2$.
		This is essentially established in \cite{jerison2016drift}.
		We provide a more general result in Appendix~\ref{app:jerison}.
	\end{remark}
	
	\begin{remark}
		The bound $\Pi(C) > 1/2$ is also closely related to the aperiodicity of the
		Markov chain defined by~$P$ \citep[see,
		e.g.,][page 47]{roberts2004general}.  Indeed, one can easily show
		that, when $\Pi(C) > 1/2$ and the minorization condition (B\hyperref[b2]{2}) holds, the
		chain must be aperiodic.
	\end{remark}
	
	Proposition~\ref{pro:jerison} shows that, under (B\hyperref[b1]{1})-(B\hyperref[b3]{3}), the size of~$C$ cannot be too small.
	This typically leads to a restriction on the minorization parameter~$\varepsilon$, which in turn puts a bound on $\rho_{\mbox{\scriptsize{opt}}}^\dagger(P)$.
	The next result, which is the analogue of
	Theorem~\ref{thm:modeloptima1}, is an immediate consequence of
	Proposition~\ref{pro:paraoptima-B} in conjunction with Proposition~\ref{pro:jerison}.
	\begin{theorem}
		\label{thm:modeloptima2}
		Let $P$ be a Mtk such that $\tilde{T}(P) \neq \emptyset$, and
		let~$\Pi$ denote the stationary distribution.  Then
		\[
		\rho_{\mbox{\scriptsize{opt}}}^\dagger(P) \geq \inf_{C \in
			\mathcal{B}, \Pi(C) > 1/2} (1-\varepsilon_C) \,,
		\]
		where $\varepsilon_C$ is defined at \eqref{eq:varepsilonC}.
	\end{theorem}
	
	Theorem~\ref{thm:modeloptima2} suggests that
	$\rho_{\mbox{\scriptsize{opt}}}^\dagger(P)$ is far away from~$1$ only
	if there exists a set~$C$ such that $\Pi(C)>1/2$ and $\varepsilon_C$
	is large.  As we shall see in the next subsection, when the state
	space is high-dimensional (or finite but large), such a~$C$ will
	often not exist, even when the associated chain mixes rapidly.
	
	\subsection{Examples}
	\label{ssec:example-2}
	
	\subsubsection{Gaussian autoregressive chain}
	
	Here we take one final look at the Gaussian autoregressive chain,
	which has Mtk given by
	\[
	P_n(x, \df y) \propto \exp \left( -\frac{2}{3} \left\| y -
	\frac{x}{2} \right\|^2 \right) \df y \,, \quad x \in
	\mathbb{R}^n \,.
	\]
	It is shown in \citepage{qin2018wasserstein}{Proposition~4} that the convergence rate
	bound in Theorem~\ref{thm:rosenbound}, which is based on
	(B\hyperref[b1]{1})-(B\hyperref[b3]{3}), converges to~$1$ as $n \to
	\infty$.  The following result, which is the analogue of
	Proposition~\ref{pro:gaussian}, shows that, in fact, the optimal bound
	also converges to 1.
	
	\begin{proposition}
		$\rho_{\mbox{\scriptsize{opt}}}^\dagger(P_n) \to 1$ as $n \to
		\infty$.
	\end{proposition}
	
	\begin{proof}
		For a given positive integer~$n$, let $C_n \subset \mathbb{R}^n$ be
		measurable, and let $D_n$ be its diameter.  Suppose that $\Pi_n(C_n) >
		1/2$.  Then $D_n^2/4$ must be larger than the median of the $\chi^2_n$
		distribution, which we denote by $m_n$.  Let $p_n(x,y)$ be the density
		function of $P_n(x,\df y)$.  
		By~\eqref{eq:gaussian-tv} in Section~\ref{sssec:gaussian1},
		\[
		\varepsilon_{C_n} \leq 2 \Phi\left(-\frac{D_n}{2\sqrt{3}} \right) \leq 2
		\Phi\left(-\sqrt{\frac{m_n}{3}} \right).
		\]
		By
		Theorem~\ref{thm:modeloptima2},
		\[
		\rho_{\mbox{\scriptsize{opt}}}^\dagger(P_n) \geq 1 - 2
		\Phi\left(-\sqrt{\frac{m_n}{3}} \right) \,.
		\]
		Since $m_n \to \infty$ as $n \to \infty$, the result follows
		immediately.
	\end{proof}
	
	\subsubsection{Metropolis adjusted Langevin algorithm}
	
	Consider again the MALA algorithm described in
	Section~\ref{sssec:mala}, and, as before, denote its Mtk and step size
	by $P_n$ and $h_n$, respectively.  The following analogue of
	Proposition~\ref{pro:mala}, whose proof is given in
	Appendix~\ref{app:mala2}, shows that any simple upper bound on
	$\rho_*(P_n)$ based on (B\hyperref[b1]{1})-(B\hyperref[b3]{3})
	converges to~$1$ at a faster than polynomial rate.
	\begin{proposition}
		\label{pro:mala2}
		Suppose that (H\hyperref[h1]{1}) and (H\hyperref[h2]{2}) are
		satisfied, and that $h_n < n^{-\gamma}$ for some $\gamma > 0$.  Then
		for any $\gamma' > 0$,
		\[
		\lim\limits_{n \to \infty} n^{\gamma'} [1 -
		\rho_{\mbox{\scriptsize{opt}}}^\dagger(P_n)] = 0 \,.
		\]
	\end{proposition}


	\subsubsection{Markov chains on graphs}
	
	Let $G = (\X,E)$ be an undirected simple graph, and consider a
	Markov chain on the state space~$\X$.  Denote its transition kernel
	by~$P$, and its stationary distribution (which is assumed to
	uniquely exist) by~$\Pi$.  For $x,y \in \X$, $P(x,\{y\}) > 0$
	only if~$x$ and~$y$ are connected by an edge.  (Note that,
	under this assumption, $P(x,\{x\}) = 0$ for each $x \in \X$.)
	Let
	\[
	m_0 = \inf_{C \subset \X, \Pi(C) > 1/2} \# (C) \,,
	\]
	where $\#(C)$ is the cardinality of $C \subset \X$, and let
	\[
	m_1 = \sup_{x \in \X} d(x) \,,
	\]
	where $d(x)$ is the degree of~$x$.  In order for the
	minorization condition~(B\hyperref[b2]{2}) to hold on a
	set~$C$, there must exist a vertex $x_C \in \X$ that's
	connected to every vertex in~$C$.  By Proposition~\ref{pro:jerison}, when $m_0 >
	m_1$, it is impossible for (B\hyperref[b1]{1})-(B\hyperref[b3]{3}) to hold
	simultaneously, so
	$\rho_{\mbox{\scriptsize{opt}}}^\dagger(P)=1$.
	
	The following is a specific example.
	
	\begin{figure}
		\centering
		\begin{subfigure}{0.27\textwidth}
			\centering
			\includegraphics[width=\textwidth]{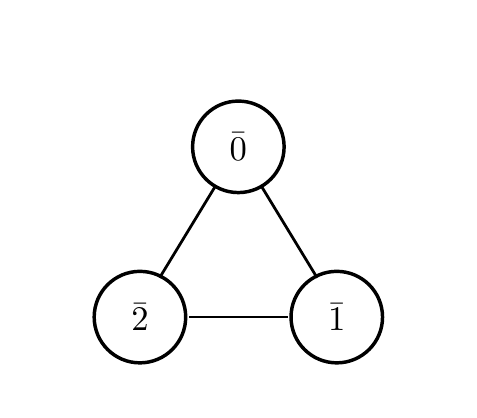}
			\caption{$\mathbb{Z}/n\mathbb{Z}$ with $n=3$.} \label{fig:circle}
		\end{subfigure}
		\qquad\qquad
		\begin{subfigure}{0.27\textwidth}
			\centering
			\includegraphics[width=\textwidth]{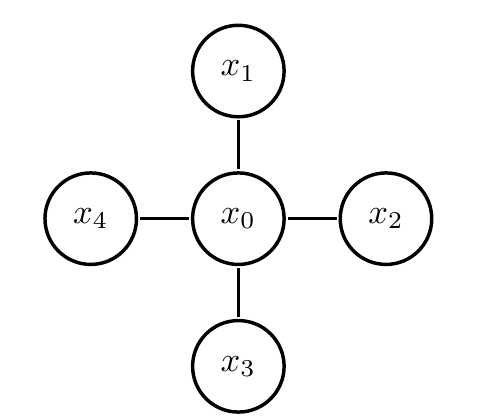}
			\caption{``Star" with $n=4$.} \label{fig:star}
		\end{subfigure}
		\caption{Visualization of the two graph examples.}\label{fig:discrete}
	\end{figure}

	\vspace{.1in} {\noindent \it Random walk on
		$\mathbb{Z}/n\mathbb{Z}$.}  Let $n \geq 3$ be an odd number,
	and let~$\X$ be $\mathbb{Z}/n\mathbb{Z} =
	\{\bar{0},\bar{1},\dots,\overline{n-1}\}$, the set of integers
	mod~$n$.  Arrange the elements of $\X$ in a loop, so that two
	vertices are connected if and only if they differ
	by~$\bar{1}$.  
	See Figure~\ref{fig:circle}.
	For each $\bar{x} \in \X$,
	$P(\bar{x},\{\bar{y}\}) = 1/2$ for $\bar{y} = \bar{x} -
	\bar{1}$ and $\bar{y} = \bar{x}+\bar{1}$.  It's easy to verify
	that the stationary distribution is uniform, $m_0 = (n+1)/2$,
	and $m_1 = 2$.  It follows that (B\hyperref[b1]{1})-(B\hyperref[b3]{3}) cannot hold simultaneously whenever $n
	\geq 5$.  On the other hand, the true convergence rate
	$\rho_*(P)$ is always strictly less than~$1$, and $1 -
	\rho_*(P)$ is of order $n^{-2}$ when~$n$ is large \citep[][page 44]{diaconis1991geometric}.
	
%
	
	Next, we provide an example where (B\hyperref[b1]{1})-(B\hyperref[b3]{3}) is capable of producing sharp bounds.
	
	\vspace{.1in} {\noindent \it Lazy random walk on a ``star".}
	Consider a graph with a central vertex, $x_0$, and~$n$ outside
	vertices, $x_1,x_2,\dots,x_n$.  Each outside vertex is
	connected to~$x_0$ but not to other vertices.  
	See Figure~\ref{fig:star}.
	For this
	example, we allow the chain to be lazy, i.e., we no longer
	assume that $P(x,\{x\}) = 0$ for each $x \in \X$.  To be
	specific, assume that
	\[
	P(x_0,\{x_i\}) = \theta 1_{i=0} + \frac{1-\theta}{n} 1_{i \neq 0} \,,
	\]
	and that, for $i \neq 0$,
	\[
	P(x_i,\{x_j\}) = \theta 1_{j=i} + (1-\theta) 1_{j=0} \,,
	\]
	where $\theta \in (0,1)$ is a constant.
	In this case, $\rho_*(P) = \max\{\theta,1-2\theta\}$, which is independent of~$n$ \citep[][page 49]{diaconis1991geometric}.
	
	It can be verified that $\Pi(\{x_0\}) = 1/2$, and $\Pi(\{x_i\}) = 1/(2n)$ for $i \neq 0$.
	Therefore, $\Pi(C) > 1/2$ if and only if $\{x_0\}$ is a proper subset of~$C$.
	It's not difficult to verify that when $C = \{x_0,x_i\}$ for some $i \neq 0$, 
	\[
	\varepsilon_C = \min\{\theta,(1-\theta)/n\} + \min\{1-\theta,\theta\}  \,.
	\]
	If $C$ contains additional elements, $\varepsilon_C$ is, of course, smaller.
	By Theorem~\ref{thm:modeloptima2}, 
	\[
	\rho_{\mbox{\scriptsize{opt}}}^\dagger(P) \geq \max\{\theta,1
	- \theta\} - \min\{\theta,(1-\theta)/n\} \,.
	\]
	This does not rule out the possibility of obtaining a sharp upper bound on $\rho_*(P)$ based on (B\hyperref[b1]{1}) to (B\hyperref[b3]{3}).
	Indeed, letting $V(x_i) = 0$ for $i = 0,1,\dots,n$, one can show that (B\hyperref[b1]{1}) holds with $\eta = 0$, $L=0$, while (B\hyperref[b2]{2}) and (B\hyperref[b3]{3}) hold with $d > 0$ and $\varepsilon = \min\{\theta, 1- \theta\}$.
	Plugging these parameters into Theorem~\ref{thm:rosenbound} yields a coupling-based upper bound on the true convergence rate $\rho_*(P)$, that is,
	\[
	1 - \varepsilon = \max\{\theta, 1- \theta\} \,.
	\]
	Note that this bound is tight whenever $\theta \geq 1/2$.
	
	\section{Discussion} \label{sec:beyondoptima}
	
	We have demonstrated that methods based on single-step d\&m can have
	serious limitations in the construction of convergence rate bounds for
	Markov chains.  On the other hand, it is certainly not the case that
	these methods are incapable of producing sharp bounds, see, e.g.,
	\citet{yang2019complexity,qin2019convergence,ekvall2019convergence}.
	It's also worth noting that we've only examined the most basic versions
	of d\&m.  
	Indeed, as mentioned in the Introduction,
	multi-step d\&m can lead to sharp
	convergence rate bounds, even in situations where single-step methods
	fail completely.  
	Unfortunately, it is almost always impossible to get
	a handle on (much less a closed form for) the multi-step Mtks
	associated with practically relevant Monte Carlo Markov chains.  
	
	More generally, one can construct coupling schemes that take
	into account the multi-step dynamics of a Markov chain, and
	these can lead to sharp convergence rate bounds.  One such
	technique is \textit{one-shot coupling}
	\citep{roberts2002one}, which can be understood as
	constructing convergence bounds with respect to some
	Wasserstein distance, and then transforming these to total
	variation bounds \citep{madras2010quantitative}.  These types
	of techniques are promising because convergence rate bounds
	with respect to Wasserstein distances are often more robust to
	increasing dimension, see, e.g.,
	\citet{hairer2014spectral,monmarche2018elementary,durmus2019high,eberle2019quantitative,qin2018wasserstein,bou2020coupling}.
	The idea of drift and minorization can also be
	employed to convergence analysis with respect to Wasserstein
	distances, and this has led to a series of useful
	results \citep[see,
	e.g.,][]{hairer2011asymptotic,butkovsky2014subgeometric,douc2018markov}.
	
	There are, of course, other families of methods for estimating the convergence rates of Markov chains.
	Most notably, for many important problems, it's possible to construct well-behaved bounds on the $L^2$ convergence rates of Markov chains using functional inequalities \citep[see, e.g.,][]{holley1988simulated,lawler1988bounds,jerrum2004elementary,cattiaux2014long}.
	
	We end this section with a brief discussion of
	the limitations of our method, which point to possible avenues
	for future work.
	The parameter-specific bound in Theorem~\ref{thm:paraoptima} is tight only for chains that are reversible and non-negative definite.
	To find good estimates of the parameter-specific optimal bound when at least one of (S\hyperref[s1]{1}) and (S\hyperref[s2]{2}) does not hold, one needs to find slowly converging chains that satisfy (A\hyperref[a1]{1})-(A\hyperref[a3]{3}) with the given parameter value, but do not satisfy both (S\hyperref[s1]{1}) and (S\hyperref[s2]{2}).
	This proves to be a challenging task.
	One limitation of Theorem~\ref{thm:modeloptima1} is that, unlike Theorem~\ref{thm:modeloptima2}, it cannot provide non-trivial bounds for chains on discrete state spaces.
	This is because $1-\varepsilon_C = 0$ whenever~$C$ is a singleton.
	It is unclear whether this can be overcome by a more delicate analysis of the chain-specific optimal bound.

	\vspace*{10mm} 
	
	\appendix

	\section{Optimal bounds for chains that do not satisfy both (S\hyperref[s1]{1}) and (S\hyperref[s2]{2})} \label{app:nonrev}
	
	
	For $(\lambda,K,\varepsilon,\beta) \in T_0$, let $S_{\lambda,K,\varepsilon,\beta}$ be the collection of Mtks that satisfy (A\hyperref[a1]{1})-(A\hyperref[a3]{3}) with this parameter value, and let $S^{(R)}_{\lambda,K,\varepsilon,\beta}$ be the subset of $S_{\lambda,K,\varepsilon,\beta}$ consisting of reversible Mtks, so that $S^{(N)}_{\lambda,K,\varepsilon,\beta} \subset S^{(R)}_{\lambda,K,\varepsilon,\beta} \subset S_{\lambda,K,\varepsilon,\beta}$.
	Then
	\[
	\rho_{\mbox{\scriptsize{opt}}}(\lambda,K,\varepsilon,\beta) = \sup_{P
		\in S_{\lambda,K,\varepsilon,\beta}} \rho_*(P) 
	\]
	is the parameter-specific optimal bound based on (A\hyperref[a1]{1})-(A\hyperref[a3]{3}) alone, and
	\[
	\rho^{(R)}_{\mbox{\scriptsize{opt}}}(\lambda,K,\varepsilon,\beta) = \sup_{P
		\in S^{(R)}_{\lambda,K,\varepsilon,\beta}} \rho_*(P) 
	\]
	is the parameter-specific optimal bound based on (A\hyperref[a1]{1})-(A\hyperref[a3]{3}) and (S\hyperref[s1]{1}).
	It's evident that
	\begin{equation} \label{ine:rho-opt-order}
	\rho^{(N)}_{\mbox{\scriptsize{opt}}}(\lambda,K,\varepsilon,\beta) \leq \rho^{(R)}_{\mbox{\scriptsize{opt}}}(\lambda,K,\varepsilon,\beta) \leq \rho_{\mbox{\scriptsize{opt}}}(\lambda,K,\varepsilon,\beta) \,.
	\end{equation}
	
	Recall that~$\rho_{\mbox{\scriptsize{opt}}}^*(P)$ is the best simple upper bound that can be constructed for an Mtk~$P$ based on (A\hyperref[a1]{1})-(A\hyperref[a3]{3}) along with (S\hyperref[s1]{1}) and (S\hyperref[s2]{2}) (if applicable), after the d\&m parameter is optimized.
	If~$P$ is reversible and non-negative definite, then, as noted in Section~\ref{ssec:modeloptima}, 
	\[
	\rho_{\mbox{\scriptsize{opt}}}^*(P) = \rho_{\mbox{\scriptsize{opt}}}^{*(N)}(P) := \inf_{(\lambda, K, \varepsilon,
		\beta) \in T(P)} \rho^{(N)}_{\mbox{\scriptsize{opt}}}(\lambda, K,
	\varepsilon, \beta) \,.
	\]
	If~$P$ is reversible but not non-negative definite, then
	\[
	\rho_{\mbox{\scriptsize{opt}}}^*(P) = \inf_{(\lambda, K, \varepsilon,
		\beta) \in T(P)} \rho^{(R)}_{\mbox{\scriptsize{opt}}}(\lambda, K,
	\varepsilon, \beta) \,.
	\]
	If~$P$ is neither reversible nor negative definite, then
	\[
	\rho_{\mbox{\scriptsize{opt}}}^*(P) = \inf_{(\lambda, K, \varepsilon,
		\beta) \in T(P)} \rho_{\mbox{\scriptsize{opt}}}(\lambda, K,
	\varepsilon, \beta) \,.
	\]
	By~\eqref{ine:rho-opt-order}, in all three cases, $\rho_{\mbox{\scriptsize{opt}}}^{*(N)}(P) \leq \rho_{\mbox{\scriptsize{opt}}}^*(P)$.

	\section{Proof of Theorem~\ref{thm:pic-1}}
	\label{app:pic-1}
	
	\begin{proof}
		(A\hyperref[a1]{1}) implies that
		\[
		PV(x) \leq \lambda V(x) + K1_C(x) \,.
		\]
		Along with the other assumptions, this drift condition
		implies that the chain is geometrically ergodic
		\citep[Theorem 15.0.1]{meyn2009markov}.  Thus,~$P$
		admits a unique stationary distribution~$\Pi$.
		
		If $\lambda = 0$, then $PV(x) \leq 0$ for each $x
		\not\in C$.  Since $V \geq 1$, this implies that $C =
		\X$, and $\Pi(C) = 1$.  In the remainder of the proof,
		we assume that $\lambda \in (0,1)$.
		
		Note that $C \neq \emptyset$.
		Otherwise, (A\hyperref[a1]{1}) implies that $P^mV(x) \leq \lambda^m V(x)$ for each positive integer~$m$ and $x \in \X$.
		This implies that $P^m V(x) < 1$ for sufficiently large~$m$, which is not possible since $V \geq 1$. 
		
		Let $\{X_m\}_{m=0}^{\infty}$ be a chain evolving according to~$P$ such that $X_0 \in C$ is fixed, and let $\{\mathcal{F}_m\}$ be its natural filtration.
		For $m \geq 1$, let $N_m = \sum_{i=0}^{m-1} 1_C(X_i)$, and set $N_0 = 0$.
		We now use a standard argument to bound $\mathbb{P}(N_m \geq i)$ for a positive integer~$i$ \citep[see, e.g.,][]{roberts2004general}.
		For any positive integer~$j$, let
		\[
		Z_j = (\lambda^{-1}K)^{-N_j-1} \lambda^{-j} V(X_j) \,.
		\]
		If $X_{j-1} \not\in C$, then $N_j = N_{j-1}$, and
		\[
		\begin{aligned}
		\mathbb{E}(Z_j|\mathcal{F}_{j-1})
		 &= (\lambda^{-1}K)^{-N_{j-1}-1} \lambda^{-j} \mathbb{E}[V(X_j)|X_{j-1}] \\
		 &\leq (\lambda^{-1}K)^{-N_{j-1}-1} \lambda^{-j} \lambda V(X_{j-1})  \\
		 &= Z_{j-1} \,.
		\end{aligned}
		\]
		If $X_{j-1} \in C$, then $N_j = N_{j-1} + 1$, and
		\[
		\begin{aligned}
		\mathbb{E}(Z_j|\mathcal{F}_{j-1}) &= (\lambda^{-1}K)^{-N_{j-1}-2} \lambda^{-j} \mathbb{E}[V(X_j)|X_{j-1}] \\
		&\leq (\lambda^{-1}K)^{-N_{j-1}-2} \lambda^{-j} K \\
		&\leq Z_{j-1} \,.
		\end{aligned}
		\]
		It follows that
		\[
		\mathbb{E}Z_m \leq \mathbb{E}Z_1 \leq (\lambda^{-1}K)^{-1} \,.
		\]
		By Markov's inequality,
		\[
		\mathbb{P}(N_m < i) \leq \mathbb{E} (\lambda^{-1}K)^{i-1-N_m} \leq (\lambda^{-1}K)^i \lambda^{m} \mathbb{E} Z_m \leq (\lambda^{-1}K)^{i-1} \lambda^{m} \,.
		\]
		For a positive integer~$m$, let $i_m = \lfloor m \log \lambda^{-1} / (\log \lambda^{-1} + \log K) \rfloor$. 
		Then $(\lambda^{-1}K)^{i_m} \leq \lambda^{-m}$. 
		It follows that
		\begin{equation} \label{eq:Nm}
		\mathbb{E}N_m \geq \sum_{i=1}^{i_m} \mathbb{P}(N_m \geq i) \geq i_m - \lambda^m \frac{(\lambda^{-1}K)^{i_m}-1}{\lambda^{-1}K - 1} \geq i_m - \frac{1-\lambda^m}{\lambda^{-1}K - 1} \,.
		\end{equation}
		The strong law of large numbers holds for any ergodic chain \citep[][Theorem~3]{tierney1994markov}.
		Therefore, $N_m/m \to \Pi(C)$ as $m \to \infty$, almost surely.
		Since $N_m/m$ is bounded, by the dominated convergence theorem, $\mathbb{E}N_m/m \to \Pi(C)$ as $m \to \infty$.
		Then~\eqref{eq:Nm} implies that
		\[
		\Pi(C) \geq \frac{\log \lambda^{-1}}{\log K + \log \lambda^{-1}}  \,.
		\]
	\end{proof}

	\section{Proof of Proposition~\ref{pro:gaussian}}
	\label{app:gaussian}
	
	\begin{proof}
		By~\eqref{ine:gaussianoptima},
		\begin{equation} \nonumber
		\rho_n^* = \inf_{D > 0} \left[ 1 - 2\Phi\left( -\frac{D}{2\sqrt{3}} \right) \right]^{1/\lfloor \alpha_n(D) \rfloor} \leq \rho_{\mbox{\scriptsize{opt}}}^*(P_n) \leq 1 \,,
		\end{equation}
		where 
		\[
		\alpha_n(D) = \left(\int_0^{D^2/4} f_n(x) \, \df x \right)^{-1} \,,
		\]
		and $f_n$ is the density function of the $\chi^2_n$
		distribution.  It suffices to show that $\rho_n^* \to
		1$ as $n \to \infty$.  Let $\delta > 0$ be arbitrary.
		There exists $M_{\delta} > 0$ such that
		\[
		1 - 2\Phi\left( -\frac{D}{2\sqrt{3}} \right) \geq 1 - \delta
		\]
		whenever $D \geq M_{\delta}$.
		Since $\alpha_n(D)$ is bounded below by~$1$,
		\begin{equation} \label{eq:DMdelta}
		\left[1 - 2\Phi\left( -\frac{D}{2\sqrt{3}} \right)
		\right]^{1/\lfloor \alpha_n(D) \rfloor} \geq 1 -
		\delta \; \quad \forall D \geq M_{\delta} \,.
		\end{equation}
		Now, let $0 < D < M_{\delta}$.
		By the mean value theorem,
		\[
		1 - 2\Phi\left( -\frac{D}{2\sqrt{3}} \right) \geq A_{\delta} D \,,
		\]
		where
		\[
		A_{\delta} = \frac{1}{\sqrt{6\pi}} \exp \left( -\frac{M_{\delta}^2}{24} \right) \,.
		\]
		The mode of $f_n(\cdot)$ is known to be $n-2$ for $n \geq 2$.
		For $n \geq M_{\delta}^2/4 + 2$,
		\[
		\frac{1}{\alpha_n(D)} = \int_0^{D^2/4} f_n(x) \df x \leq \frac{ f_n(M_{\delta}^2/4) }{4} D^2 \,.
		\]
		It follows that, when $n \geq M_{\delta}^2/4 + 2$,
		\[
		\frac{1}{\lfloor \alpha_n(D) \rfloor} \leq \left\lfloor \frac{4}{f_n(M_{\delta}^2/4) D^2} \right\rfloor^{-1} \,.
		\]
		Recall that, for each $x > 0$,
		\[
		f_n(x) = \frac{1}{2^{n/2} \Gamma(n/2)} x^{n/2-1} e^{-x/2} \,.
		\]
		It's easy to verify that, as $n \to \infty$, $f_n(x) \to 0$.
		It follows that there exists $N_{\delta} > M_{\delta}^2/4+2$ such that, for each $n \geq N_{\delta}$,
		\[
		\left\lfloor \frac{4}{f_n(M_{\delta}^2/4) D^2} \right\rfloor^{-1} \leq \left( \frac{4}{f_n(M_{\delta}^2/4) D^2} - 1 \right)^{-1}  \leq \frac{f_n(M_{\delta}^2/4)D^2}{4 - f_n(M_{\delta}^2/4) M_{\delta}^2} \leq B_{\delta} D^2 \,,
		\]
		where 
		\[
		B_{\delta} = -2eA_{\delta}^2 \log(1-\delta) \,.
		\]
		Thus, for $n \geq N_{\delta}$ (and $D \in (0,M_{\delta})$),
		\[
		\left[ 1 - 2\Phi\left( -\frac{D}{2\sqrt{3}} \right) \right]^{1/\lfloor \alpha_n(D) \rfloor} \geq (A_{\delta}D)^{B_{\delta}D^2} \,.
		\]
		Taking the derivative shows that $B_{\delta} x^2 \log (A_{\delta} x)$ is minimized when $x = e^{-1/2} A_{\delta}^{-1}$, and
		\[
		\inf_{x > 0} (A_{\delta}x)^{B_{\delta}x^2} = 1 - \delta \,.
		\]
		Combining this fact with~\eqref{eq:DMdelta} shows that
		\[
		\inf_{D > 0} \left[ 1 - 2\Phi\left( -\frac{D}{2\sqrt{3}} \right) \right]^{1/\lfloor \alpha_n(D) \rfloor} \geq 1-\delta
		\]
		whenever $n \geq N_{\delta}$.
		This completes the proof.
	\end{proof}
	
	\section{Proof of Proposition~\ref{pro:mala}}
	\label{app:mala}
	
	\begin{proof}
		The proof is an application of
		Theorem~\ref{thm:modeloptima1}.  For a positive
		integer~$n$, let $C_n$ be a measurable subset of
		$\mathbb{R}^n$ such that $\Pi_n(C_n) > 0$.  Since
		$\Pi_n$ admits a density function, the diameter of
		$C_n$, which we denote by $D_n$, must be
		non-vanishing.  It's clear that $\Pi_n(C_n) \leq
		(GD_n)^n$, where $G = \sup_{x \in \mathbb{R}} g(x) <
		\infty$, and $g: \mathbb{R} \to [0,\infty)$ is defined
		in Condition (H\hyperref[h1]{1}).

		Next, we bound $\varepsilon_{C_n}$. 
		By~\eqref{eq:malakernel}, the transition kernel of a MALA chain
		can be written as
		\[
		P_n(x,\df y) = a_n(x,y) q_n(x,y) \, \df y + \int_{\X}
		\big[ 1-a_n(x,z) \big] q_n(x,z) \, \df z \,
		\delta_x(\df y) \,, \quad x \in \X \,,
		\]
		where $q_n(x,\cdot)$ is the probability density
		function of the $\mbox{N}_n(x-h_n\nabla f_n(x),
		2h_nI_n)$ distribution, and $a_n(\cdot,\cdot) \leq 1$
		is defined as in~\eqref{eq:malakernel}.  Suppose that
		(A\hyperref[a2]{2}) holds on $C_n$ with $\varepsilon_n
		> 0$ and probability measure $\nu_n$.  Let $x,x' \in
		C_n$ be such that $x \neq x'$.  
		(A\hyperref[a2]{2}) implies that
		\begin{equation} \label{ine:mala-minor-1}
		a_n(x,y) q_n(x,y) \, \df y + \int_{\X} \big[
		1-a_n(x,z) \big] q_n(x,z) \, \df z \, \delta_x(\df
		y) \geq \varepsilon_n \nu_n(\df y) \;,
		\end{equation}
		and
		\begin{equation} \label{ine:mala-minor-2}
		a_n(x',y) q_n(x',y) \, \df y + \int_{\X} \big[
		1-a_n(x',z) \big] q_n(x',z) \, \df z \,
		\delta_{x'}(\df y) \geq \varepsilon_n \nu_n(\df y)
		\,.
		\end{equation}
		Integrating both sides of~\eqref{ine:mala-minor-1} over the set $\{x'\}$ shows that $\nu_n(\{x'\}) = 0$.
		Analogously, $\nu_n(\{x\}) =
		0$.
		Then~\eqref{ine:mala-minor-1} and~\eqref{ine:mala-minor-2} imply that the following hold almost everywhere:
		\[
		q_n(x,y) \geq \varepsilon_n \nu_n(\df y)/\df y \;\;\;
		\mbox{and} \;\;\; q_n(x',y) \geq \varepsilon_n \nu_n(\df
		y)/ \df y \,.
		\]
		By an argument that was used in Section~\ref{sssec:gaussian1},
		\[
		\begin{aligned}
		\varepsilon_n &\leq \int_{\X} \min \{ q_n(x,y),
		q_n(x',y) \} \, \df y \\
		&= 2 \Phi \left(
		-\frac{\|x-h_n\nabla f_n(x) - x' + h_n\nabla
			f_n(x')\|}{2\sqrt{2h_n}} \right) \,.
		\end{aligned}
		\]
		Recall that, under (H\hyperref[h1]{1}) and (H\hyperref[h2]{2}), $f_n$ is $M$-smooth.
		Thus, whenever $h_n M < 1$,
		\[
		\|x-h_n\nabla f_n(x) - x' + h_n\nabla f_n(x')\| \geq (1-h_n M) \|x-x'\| \,.
		\]
		Letting~$x$ and $x'$ vary in~$C_n$ shows that
		\[
		\varepsilon_{C_n} \leq 2 \Phi\left( -\frac{(1-h_n M)D_n}{2\sqrt{2h_n}} \right) \,.
		\]
		By Theorem~\ref{thm:modeloptima1} and the assumption that $h_n < n^{-\gamma}$, for $n$ large enough, $1 - h_nM > 1/\sqrt{2}$, and
		\begin{equation} \label{ine:malarate0}
		\begin{aligned}
		\rho_{\mbox{\scriptsize{opt}}}^*(P_n) &\geq \inf_{D > 0} \left[ 1 - 2 \Phi\left( -\frac{(1-h_n M)D}{2\sqrt{2h_n}} \right) \right]^{\lfloor (G D)^{-n} \rfloor^{-1} \wedge 1} \\
		&\geq \inf_{D > 0} \left[ 1 - 2 \Phi\left( -D n^{\gamma/2} /4 \right) \right]^{\lfloor (G D)^{-n} \rfloor^{-1} \wedge 1} \,,
		\end{aligned}
		\end{equation}
		where $a \wedge b$ means $\min \{a,b\}$.
		When $D > 1/(2G)$,
		\begin{equation} \label{ine:malarate1}
		\left[ 1 - 2 \Phi\left( -D n^{\gamma/2} /4 \right) \right]^{\lfloor (G D)^{-n} \rfloor^{-1} \wedge 1} \geq 1 - 2 \Phi(-n^{\gamma/2}/(8G)) \,.
		\end{equation}
		When $D \leq 1/(2G)$, $\lfloor (GD)^{-n} \rfloor \geq (2^{n-1}-1)(GD)^{-1}$.
		This implies that, whenever $D \leq 1/(2G)$ and $n > 1$,
		\begin{equation} \label{ine:malarete2}
		\begin{aligned}
		\left[ 1 - 2 \Phi\left( -D n^{\gamma/2} /4 \right)
		\right]^{\lfloor (G D)^{-n} \rfloor^{-1} \wedge 1}
		&\geq \left[ 1 - 2 \Phi\left( -D n^{\gamma/2} /4
		\right) \right]^{GD/(2^{n-1}-1)} \\ &\geq \left\{
		\inf_{x > 0} \big[ 1-2\Phi(-x/4) \big]^{Gx}
		\right\}^{1/(2^{n-1}n^{\gamma/2}-n^{\gamma/2})} .
		\end{aligned}
		\end{equation}
		To proceed, we study the asymptotic behavior of the
		right-hand sides of \eqref{ine:malarate1} and
		\eqref{ine:malarete2}.  One can verify that, for any
		$\gamma' > 0$, as $n \to \infty$, $n^{\gamma'}
		\Phi(-n^{\gamma/2}/(8G)) \to 0$.  Moreover,
		\[
		c = \inf_{x > 0} [1-2\Phi(-x/4)]^{Gx} > 0 \,,
		\]
		which implies that, for any $\gamma' > 0$,
		\[
		\lim\limits_{n \to \infty} n^{\gamma'} \left[ 1- c^{1/(2^{n-1}n^{\gamma/2}-n^{\gamma/2})} \right] = 0 \,.
		\]
		The proof is completed by combining \eqref{ine:malarate0}, \eqref{ine:malarate1}, and \eqref{ine:malarete2}.
	\end{proof}
	
	\section{Proof of an Assertion in Remark~\ref{rem:jerison}} \label{app:jerison}
	For a set $\tilde{C} \in \mathcal{B} \times \mathcal{B}$, let
	\[
	\tilde{C}_1 = \{x \in \X: (x,y) \in \tilde{C} \}, \quad \tilde{C}_2 = \{y \in \X: (x,y) \in \tilde{C} \}
	\]
	be its projections.
	The following is a generalization of \pcite{jerison2016drift} Proposition 2.16.
	\begin{proposition} \label{pro:bivariate}
		Let~$P$ be a transition kernel that admits a
		stationary distribution~$\Pi$.  Suppose that there
		exist $\lambda' < 1$, $K' \in (0,\infty)$, $\tilde{C} \in
		\mathcal{B} \times \mathcal{B}$, and measurable functions $V_1: \X \to
		[0,\infty)$ and $V_2: \X \to [0,\infty)$ such
		that $\inf_{x,y \in \X} [V_1(x) + V_1(y)]
		> 0$, $\tilde{C}_i \in \mathcal{B}$ for $i=1,2$, and
		\begin{equation} \nonumber
		PV_1(x) + P V_2(y) \leq \lambda' [V_1(x) + V_2(y)]
		1_{(\X \times \X) \setminus \tilde{C}}(x,y) + K'
		1_{\tilde{C}}(x,y)
		\end{equation}
		for each $x,y \in \X$.
		Then $\Pi(\tilde{C}_1)+\Pi(\tilde{C}_2) > 1$.
	\end{proposition}
	\begin{proof}
		Let $\tilde{\Pi}$ be a probability measure on $(\X \times \X, \mathcal{B} \times \mathcal{B})$ such that, for each $A \in \mathcal{B}$,
		\begin{equation} \label{eq:tildepi}
		\tilde{\Pi}(A \times \X) = \tilde{\Pi}(\X \times A) = \Pi(A) \,.
		\end{equation}
		In other words, $\tilde{\Pi}$ is the joint distribution of some random element $(X,Y)$ such that, marginally,~$X$ and~$Y$ are distributed as~$\Pi$.
		By assumption, for each $x,y \in \X$,
		\[
		PV_1(x) + PV_2(y) \leq \lambda' V_1(x) + \lambda' V_2(y) + K' 1_{\tilde{C}}(x,y) \,.
		\]
		Since~$\Pi$ is the stationary distribution, taking expectations on both sides with respect to $\tilde{\Pi}(\df x, \df y)$ yields
		\[
		\Pi V_1 + \Pi V_2 \leq \lambda' \Pi V_1 + \lambda' \Pi V_2 + K' \tilde{\Pi}(\tilde{C}) \,,
		\]
		where $\Pi V_i = \int_{\X} V_i(x) \Pi(\df x)$.
		By Proposition 4.24 in \cite{hairer2006ergodic}, $\Pi V_1 + \Pi V_2 < \infty$.
		It follows that
		\[
		K' \tilde{\Pi}(\tilde{C}) \geq (1-\lambda') (\Pi V_1 + \Pi V_2) \geq (1-\lambda') \inf_{x,y \in \X} [V_1(x) + V_1(y)] > 0 \,.
		\]
		Since $K' < \infty$, $\tilde{\Pi}(\tilde{C}) > 0$.
		
		It's now straightforward to prove the result by
		contradiction.  Suppose that $\Pi(\tilde{C}_1)+\Pi(\tilde{C}_2) \leq 1$, and note that $\tilde{C} \subset \tilde{C}_1 \times \tilde{C}_2$.  
		We
		only need to construct a joint
		distribution~$\tilde{\Pi}$ such
		that~\eqref{eq:tildepi} holds and $\tilde{\Pi}(\tilde{C}_1 \times \tilde{C}_2) = 0$. 
		If $\Pi(\tilde{C}_1) = 1$ or $\Pi(\tilde{C}_2) = 1$, then we only need to set $\tilde{\Pi}(\df x, \df y) = \Pi(\df x) \Pi(\df y)$.
		Suppose that $\Pi(\tilde{C}_1)$ and $\Pi(\tilde{C}_2)$ are both strictly less than~$1$.
		Define~$\tilde{\Pi}$ as follows:
		\[
		\begin{aligned}
		&\tilde{\Pi}(\df x, \df y) = \Pi(\df x) \Pi(\df y) \times \\
		& \left\{ \frac{1_{\tilde{C}_1 \times (\X \setminus \tilde{C}_2)}(x,y)}{1-\Pi(\tilde{C}_2)}  + \frac{1_{(\X \setminus \tilde{C}_1) \times \tilde{C}_2}(x,y)}{1-\Pi(\tilde{C}_1)} + \frac{[1-\Pi(\tilde{C_1}) - \Pi(\tilde{C}_2)] 1_{(\X \setminus \tilde{C}_1) \times (\X \setminus \tilde{C}_2)}(x,y) }{[1-\Pi(\tilde{C}_1)][1-\Pi(\tilde{C}_2)]} \right\}  \,.
		\end{aligned}
		\]
		It's easy to verify that~$\tilde{\Pi}$ is a probability measure that satisfies~\eqref{eq:tildepi} and that $\tilde{\Pi}(\tilde{C}_1 \times \tilde{C}_2) = 0$.
		This completes the proof.
	\end{proof}
	
	The bivariate drift condition in Proposition~\ref{pro:bivariate} is very common among works that utilize d\&m and coupling.
	Usually,~$V_1$ and~$V_2$ are taken to be the same univariate drift function, and $\tilde{C}$ is either the Cartesian product of a small set with itself, or a level set of the additive bivariate drift function $V_1(x)+V_2(y)$.
	In particular, taking $V_1 = V_2$ and $\tilde{C} = C' \times C'$ yields (C\hyperref[c1]{1}) in Remark~\ref{rem:jerison}.
	
	\section{Proof of Proposition~\ref{pro:mala2}}
	\label{app:mala2}
	
	\begin{proof}
		The proof is an application of
		Theorem~\ref{thm:modeloptima2}.  Let~$n$ be fixed.
		Let $C_n \subset \mathbb{R}^n$ be measurable, and
		denote its diameter by $D_n$.  Note that $\Pi_n(C_n)
		\leq (GD_n)^n$, where $G = \sup_{x \in \mathbb{R}}
		g(x) < \infty$.  When $\Pi_n(C_n) > 1/2$,
		\begin{equation} \label{ine:mala2dn}
		D_n \geq \frac{1 }{2^{1/n}G} \geq \frac{1}{2G}  \,.
		\end{equation}
		Recall that $M > 0$ is a constant defined in Condition
		(H\hyperref[h2]{2}).  In the proof of
		Proposition~\ref{pro:mala} it is shown that, whenever
		$h_nM < 1$,
		\[
		\varepsilon_{C_n}  \leq 2\Phi \left( -\frac{(1-h_n M)D_n}{2\sqrt{2}h_n} \right) \,.
		\]
		If $\Pi_n(C_n) > 1/2$ and thus,~\eqref{ine:mala2dn} holds, then
		\[
		1 - \varepsilon_{C_n} \geq 1 - 2 \Phi \left( - \frac{1-h_nM}{4\sqrt{2h_n}G} \right) \,.
		\]
		Recall that $h_n < n^{-\gamma}$.
		By Theorem~\ref{thm:modeloptima2}, for sufficiently large~$n$, $h_n M \leq 1/\sqrt{2}$, and
		\[
		\rho_{\mbox{\scriptsize{opt}}}^*(P_n) \geq 1 - 2 \Phi \left( - \frac{1-h_n M}{4\sqrt{2h_n}G} \right) \geq 1 - 2 \Phi \left( - \frac{n^{\gamma/2}}{8G} \right) \,.
		\]
		The asymptotic behavior of $\Phi(x)$ as $x \to -
		\infty$ ensures the desired result.
	\end{proof}
	
	\bigskip
	
	\noindent{\bf Acknowledgment.}\, The second author was supported by NSF Grant DMS-15-11945.

	\bibliographystyle{ims} \bibliography{limitationbib}

\end{document}